\documentclass[a4paper,12pt]{amsart}

\usepackage{amsmath,amsthm,amssymb}
\usepackage[dvipdfmx]{graphicx}
\usepackage{here}
\usepackage{color}
\usepackage[all]{xy}
\xyoption{pdf}
\usepackage[top=30truemm,bottom=30truemm,left=25truemm,right=25truemm]{geometry}
\usepackage{ytableau}
\usepackage[dvipdfmx]{hyperref}

\newtheorem{thm}{Theorem}[subsection]
\newtheorem{lem}[thm]{Lemma}
\newtheorem{prop}[thm]{Proposition}
\newtheorem{cor}[thm]{Corollary}

\theoremstyle{definition}
\newtheorem{defi}[thm]{Definition}
\newtheorem{rem}[thm]{Remark}
\newtheorem{ex}[thm]{Example}

\numberwithin{equation}{section}

\newcommand{\Par}{\mathrm{Par}}
\newcommand{\col}{\mathrm{col}}
\newcommand{\Tab}{\mathrm{Tab}}
\newcommand{\sh}{\mathrm{sh}}
\newcommand{\br}{\operatorname{br}}
\newcommand{\SpT}{Sp\mathrm{T}}
\newcommand{\SST}{\mathrm{SST}}
\newcommand{\red}{\operatorname{red}}
\newcommand{\Rem}{\operatorname{rem}}
\newcommand{\PAII}{P^{A\mathrm{II}}}
\newcommand{\QAII}{Q^{A\mathrm{II}}}
\newcommand{\Rec}{\mathrm{Rec}}
\newcommand{\LRAII}{\mathrm{LR}^{A\mathrm{II}}}
\newcommand{\LRSp}{\mathrm{LRT}^{Sp}}
\newcommand{\res}{\operatorname{res}}
\newcommand{\suc}{\operatorname{suc}}

\author{Hideya Watanabe}
\address{(H. Watanabe) Department of Mathematics, Rikkyo University, 3-34-1, Nishi-Ikebukuro, Toshima-ku, Tokyo, 171-8501, Japan}
\email{watanabehideya@gmail.com}
\date{\today}

\title{Symplectic tableaux and quantum symmetric pairs}

\begin{document}
\maketitle

\begin{abstract}
  We provide a new branching rule from the general linear group $GL_{2n}(\mathbb{C})$ to the symplectic group $Sp_{2n}(\mathbb{C})$ by establishing a simple algorithm which gives rise to a bijection from the set of semistandard tableaux of a fixed shape to a disjoint union of several copies of sets of symplectic tableaux of various shapes.
  The algorithm arises from representation theory of a quantum symmetric pair of type $A\mathrm{II}_{2n-1}$, which is a $q$-analogue of the classical symmetric pair $(\mathfrak{gl}_{2n}(\mathbb{C}), \mathfrak{sp}_{2n}(\mathbb{C}))$.
\end{abstract}

\section{Introduction}
\subsection{Branching rules}
Let $G$ be a group and $\hat{G}$ a complete set of representatives of the equivalence classes of certain irreducible $G$-modules.
Let $H$ be a subgroup of $G$.
It is a fundamental problem to determine how a given irreducible $G$-module $V \in \hat{G}$ decomposes into irreducible $H$-submodules (if it does):
\[
  V \simeq \bigoplus_{W \in \hat{H}} W^{m_{V, W}}.
\]
An explicit description of the multiplicities $m_{V,W}$ is called a \emph{branching rule}.

The problem of finding branching rules for certain pairs $(G,H)$ of classical groups (the general/special linear groups $GL_m(\mathbb{C})$, $SL_m(\mathbb{C})$, symplectic groups $Sp_{2n}(\mathbb{C})$, and (special) orthogonal groups $O_m(\mathbb{C})$, $SO_m(\mathbb{C})$) has been studied for a long time, and several (partial) answers have been obtained (see \cite{HTW05} and references therein).

In the present paper, we focus on the irreducible polynomial representations for the pair $(G, H) = (GL_{2n}(\mathbb{C}), Sp_{2n}(\mathbb{C}))$.
The equivalence classes of irreducible polynomial representations of $GL_{2n}(\mathbb{C})$ (resp., $Sp_{2n}(\mathbb{C})$) are parametrized by the set $\Par_{\leq 2n}$ (resp., $\Par_{\leq n}$) of partitions of length at most $2n$ (resp., $n$).
For each $\lambda \in \Par_{\leq 2n}$ and $\nu \in \Par_{\leq n}$, let $m_{\lambda,\nu}$ denote the corresponding multiplicity.

Littlewood \cite{Lit50} provided a partial branching rule.
Namely, he determined the multiplicities $m_{\lambda,\nu}$ for all $\lambda,\nu \in \Par_{\leq n}$, but not for all $\lambda \in \Par_{\leq 2n}$.

Sundaram \cite{S90} gave a complete branching rule.
The key ingredients for her theorem are King's symplectic tableaux (\cite{Ki76}), Berele's insertion scheme for $Sp_{2n}(\mathbb{C})$ (\cite{Ber86}), and Sundaram's algorithm.
In her branching rule, the multiplicities are determined by counting certain tableaux, which we call \emph{symplectic Littlewood-Richardson tableaux}.

Naito and Sagaki \cite{NS05} proposed a conjectural branching rule in terms of Littelmann paths.
The conjecture was proved by Schumann and Torres \cite{ScTo18}.

Kwon \cite{Kwo18b} established branching rules for various paris including $(GL_{2n}(\mathbb{C}), Sp_{2n}(\mathbb{C}))$.
His argument is based on a combinatorial crystal model, called the spinor model, and a general theory of reductive pairs.

\subsection{Results}
In the present paper, we introduce a simple algorithm which gives rise to a bijection
\[
  \LRAII: \SST_{2n}(\lambda) \rightarrow \bigsqcup_{\substack{\nu \in \Par_{\leq n} \\ \nu \subseteq \lambda}} \SpT_{2n}(\nu) \times \Rec_{2n}(\lambda/\nu)
\]
which sends a semistandard Young tableau of shape $\lambda \in \Par_{\leq 2n}$ with entries in $[1,2n] := \{ 1,\dots,2n \}$ to a pair consisting of a symplectic tableau of some shape $\nu \in \Par_{\leq n}$ with entries in $[1,2n]$ and a tableau, called a \emph{recording tableau}, of skew shape $\lambda/\nu$.
As a byproduct, it turns out that the multiplicity $m_{\lambda,\nu}$ coincides with the number $|\Rec_{2n}(\lambda/\nu)|$ of recording tableaux of shape $\lambda/\nu$.
Therefore, our algorithm provides a new branching rule for $(GL_{2n}(\mathbb{C}), Sp_{2n}(\mathbb{C}))$.
Moreover, it has deep representation theoretical information as we will see in the next subsection.
We call the bijection the \emph{Littlewood-Richardson map} since it can be regarded as a generalization of the branching rule, known as the Littlewood-Richardson rule, for the pair $(GL_{m}(\mathbb{C}) \times GL_{m}(\mathbb{C}), GL_{m}(\mathbb{C}))$.

Let us briefly explain our algorithm.
Given $T \in \SST_{2n}(\lambda)$, let $\mathbf{a} = (a_1,\dots,a_l)$ denote the first column of $T$ (read from top to bottom), and $S$ the other part.
For the column $\mathbf{a}$, define a new column $\red(\mathbf{a})$ to be the one obtained from $\mathbf{a}$ by removing the entries in the set $\Rem(\mathbf{a})$, which is defined by the following recursive formula:
\[
  \Rem(\mathbf{a}) := \begin{cases}
    \emptyset & \text{ if } l \leq 1, \\
    \Rem(a_1,\dots,a_{l-2}) \sqcup \{ a_{l-1}, a_l \} & \text{ if } l \geq 2,\ a_l \in 2\mathbb{Z},\ a_{l-1} = a_l-1, \text{ and } \\
    &\ a_l < 2l - |\Rem(a_1,\dots,a_{l-2})| - 1, \\
    \Rem(a_1,\dots,a_{l-1}) & \text{ otherwise}.
  \end{cases}
\]
Then, define a new tableau $\suc(T)$ to be the product $\red(\mathbf{a}) * S$ (in the plactic monoid).

Set $P^0 := T$, $\nu^0 := \lambda$, and $Q^0$ to be the unique tableau of shape $\lambda/\lambda$.
For each $k \geq 0$, set $P^{k+1} := \suc(P^k)$, $\nu^{k+1}$ to be the shape of $P^{k+1}$, and $Q^{k+1}$ to be the tableau of shape $\lambda/\nu^{k+1}$ such that
\[
  Q^{k+1}(i,j) = \begin{cases}
    Q^k(i,j) & \text{ if } (i,j) \notin D(\nu^k), \\
    k+1 & \text{ if } (i,j) \in D(\nu^k),
  \end{cases}
\]
where $D(\nu^k)$ denotes the Young diagram of the partition $\nu^k$.
It turns out that this procedure eventually terminates; there exists a unique integer $k_0 \geq 0$ such that
\[
  P^k = P^{k_0}, \ \nu^k = \nu^{k_0}, \text{ and } Q^k = Q^{k_0} \ \text{ for all } k \geq k_0.
\]
Set
\[
  \PAII(T) := P^{k_0}, \quad \QAII(T) := Q^{k_0}.
\]
Now, the output $\LRAII(T)$ of the algorithm is the pair $(\PAII(T), \QAII(T))$:
\[
  \LRAII(T) = (\PAII(T), \QAII(T)).
\]

\begin{ex}\label{ex: LR}\normalfont
  Let $n = 3$, $\lambda = (4,3,2,2,1)$, and consider the semistandard tableau $T$ of shape $\lambda$ given by
  \[
    T = \ytableausetup{smalltableaux}
    \begin{ytableau}
      1 & 1 & 2 & 4 \\
      2 & 2 & 3 \\
      4 & 4 \\
      5 & 6 \\
      6
    \end{ytableau}
  \]
  Then, we have
  \[
    \Rem(1,2,4,5,6) = \{ 1,2,5,6 \},
  \]
  and
  \[
    P^1 = 
    \begin{ytableau}
      4
    \end{ytableau} *
    \begin{ytableau}
      1 & 2 & 4 \\
      2 & 3 \\
      4 \\
      6
    \end{ytableau} =
    \begin{ytableau}
      1 & 2 & 4 \\
      2 & 3 \\
      4 & 4 \\
      6
    \end{ytableau}, \quad
    Q^1 = 
    \begin{ytableau}
      \none & \none & \none & 1 \\
      \none & \none & 1 \\
      \none & \none \\
      \none & 1 \\
      1
    \end{ytableau}
  \]
  The next step is computed as follows: We have
  \[
    \Rem(1,2,4,6) = \{ 1,2 \},
  \]
  and
  \[
    P^2 = 
    \begin{ytableau}
      4 \\
      6
    \end{ytableau} *
    \begin{ytableau}
      2 & 4 \\
      3 \\
      4
    \end{ytableau} =
    \begin{ytableau}
      2 & 4 & 4 \\
      3 \\
      4 \\
      6
    \end{ytableau}, \quad
    Q^2 = 
    \begin{ytableau}
      \none & \none & \none & 1 \\
      \none & 2 & 1 \\
      \none & 2 \\
      \none & 1 \\
      1
    \end{ytableau}
  \]
  Let us proceed to the next step:
  We have
  \[
    \Rem(2,3,4,6) = \{ 3,4 \},
  \]
  and
  \[
    P^3 = 
    \begin{ytableau}
      2 \\
      6
    \end{ytableau} *
    \begin{ytableau}
      4 & 4
    \end{ytableau} =
    \begin{ytableau}
      2 & 4 & 4 \\
      6
    \end{ytableau}, \quad
    Q^3 = 
    \begin{ytableau}
      \none & \none & \none & 1 \\
      \none & 2 & 1 \\
      3 & 2 \\
      3 & 1 \\
      1
    \end{ytableau}
  \]
  Since $\Rem(2,6) = \emptyset$, the algorithm now terminates.
  Hence, we finally obtain
  \[
    \LRAII(T) = \left( 
      \begin{ytableau}
        2 & 4 & 4 \\
        6
      \end{ytableau}, 
      \begin{ytableau}
        \none & \none & \none & 1 \\
        \none & 2 & 1 \\
        3 & 2 \\
        3 & 1 \\
        1
      \end{ytableau} \right).
  \]
\end{ex}

\subsection{Quantum symmetric pairs}
Although the author believes that the bijectivity of the Littlewood-Richardson map can be proved in the realm of combinatorics, we prove it via representation theory of a quantum symmetric pair of type $A\mathrm{II}_{2n-1}$, from which the algorithm to compute the Littlewood-Richardson map arises.

A quantum symmetric pair is a quantum analogue of a classical symmetric pair (e.g., $(\mathfrak{gl}_{2n}(\mathbb{C}), \mathfrak{sp}_{2n}(\mathbb{C}))$, where $\mathfrak{gl}_{2n}(\mathbb{C})$ (resp., $\mathfrak{sp}_{2n}(\mathbb{C})$) denotes the general linear Lie algebra (resp., the symplectic Lie algebra)).
It consists of a Drinfeld-Jimbo quantum group $\mathbf{U}$ and a Letzter $\imath$quantum group $\mathbf{U}^\imath$.
We refer the reader unfamiliar with quantum symmetric pair to a survey paper \cite{Wan21}.
In the present paper, we consider only a quantum symmetric pair of type $A\mathrm{II}_{2n-1}$; the $\mathbf{U}$ and $\mathbf{U}^\imath$ are quantum analogues of the universal enveloping algebras of $\mathfrak{gl}_{2n}(\mathbb{C})$ and $\mathfrak{sp}_{2n}(\mathbb{C})$, respectively.

For each $\lambda \in \Par_{\leq 2n}$, there exists a finite-dimensional irreducible $\mathbf{U}$-module $V(\lambda)$ with a distinguished basis of the form $\{ b_T \mid T \in \SST_{2n}(\lambda) \}$, called the canonical basis (\cite{Lus93}).
Similarly, for each $\nu \in \Par_{\leq n}$, there exists a finite-dimensional irreducible $\mathbf{U}^\imath$-module $V^\imath(\nu)$ (\cite{Mol06}, \cite{W21a}).
In the present paper, we prove that the $\mathbf{U}^\imath$-module $V^\imath(\nu)$ admits a distinguished basis of the form $\{ b^\imath_T \mid T \in \SpT_{2n}(\nu) \}$.

Then, for each $\lambda \in \Par_{\leq 2n}$, by composing several $\mathbf{U}^\imath$-module homomorphisms, we construct a $\mathbf{U}^\imath$-module isomorphism
\[
  \LRAII: V(\lambda) \rightarrow \bigoplus_{\nu \in \Par_{\leq n}} (V^\imath(\nu) \otimes \mathbb{Q}(q)\Rec_{2n}(\lambda/\nu)),
\]
which we call the \emph{quantum Littlewood-Richardson map}.
Here, each $x \in \mathbf{U}^\imath$ acts on each summand of the right-hand side as $x \otimes \mathrm{id}$.
The isomorphism is a $q$-analogue of the Littlewood-Richardson map on $\SST_{2n}(\lambda)$ in the following sense:
For each $T \in \SST_{2n}(\lambda)$, we have $\LRAII(b_T) \equiv b^\imath_{\PAII(T)} \otimes \QAII(T)$ modulo $q^{-1}$ (in a suitable sense).
This fact implies the bijectivity of the Littlewood-Richardson map.

Therefore, the Littlewood-Richardson map tells us not only the multiplicities $m_{\lambda,\nu}$ but also how the irreducible $\mathbf{U}$-module $V(\lambda)$ decomposes into irreducible $\mathbf{U}^\imath$-submodules at $q = \infty$.
Hence, this result must be closely related to the theory of crystal bases for quantum symmetric pairs ({\it cf.\ }\cite{W21}).

\subsection{Organization}
The paper is organized as follows.
In Section \ref{sect: preliminary comb}, we collect terminology and basic results concerning partitions and tableaux which are necessary to formulate our main algorithm.
Then, we state our main theorem in Section \ref{sect: main}.
It consists of the bijectivity of the Littlewood-Richardson map and an explicit description of the recording tableaux.
The rest of the paper is devoted to proving the theorem.
In Section \ref{sect: factor red}, we factor the reduction map ($\mathbf{a} \mapsto \red(\mathbf{a})$) into small pieces so that we can prove its injectivity.
After reviewing representation theory of $\mathfrak{gl}_{2n}(\mathbb{C})$, $\mathfrak{sp}_{2n}(\mathbb{C})$, $\mathbf{U}$, and $\mathbf{U}^\imath$ in Sections \ref{sect: preliminary Lie alg} and \ref{sect: preliminary qsp}, we prove the surjectivity of the Littlewood-Richardson map via an investigation into the quantum Littlewood-Richardson map in Section \ref{sect: LR}.
We finally complete the proof of our main theorem in Section \ref{sect: rec} by relating the recording tableaux to the symplectic Littlewood-Richardson tableaux.

\subsection*{Acknowledgements}
The author would like to thank the anonymous referees for careful readings and valuable comments.
This work was supported by JSPS KAKENHI Grant Number JP22KJ2603.

\subsection*{Notation}
Throughout the paper, we fix positive integers $m$ and $n$.

Given two integers $a$ and $b$, let $[a,b]$ denote the integer interval:
\[
  [a,b] := \{ c \in \mathbb{Z} \mid a \leq c \leq b \}.
\]

\section{Preliminaries from combinatorics}\label{sect: preliminary comb}
In this section, we collect terminology and basic results concerning partitions and tableaux which are necessary to formulate our main algorithm in the next section.

\subsection{Partitions}
A \emph{partition} is a non-increasing sequence
\[
  \lambda = (\lambda_1,\dots,\lambda_l)
\]
of positive integers.
We often regard a non-increasing sequence of nonnegative integers as a partition by ignoring the zero's.
Each $\lambda_i$ is referred to as a \emph{part} of $\lambda$.
It is convenient to set $\lambda_i := 0$ for $i > l$.
The sum $\sum_{i=1}^{l} \lambda_i$ of parts is called the \emph{size} of $\lambda$, and is denoted by $|\lambda|$.
The number $l$ of parts of $\lambda$ is called the \emph{length} of $\lambda$, and is denoted by $\ell(\lambda)$.
We regard the empty sequence $()$ as a unique partition of length $0$.
Let $\Par$ denote the set of all partitions.

For each $l \in \mathbb{Z}_{\geq 0}$, let $\mathrm{Par}_{\leq l}$ denote the set of all partitions of length at most $l$.

For each $l \in \mathbb{Z}_{\geq 0}$, let $\varpi_l$ denote the partition of length $l$ whose parts are all $1$:
\begin{align}\label{eq: varpi_l}
  \varpi_l = (1^l) = (\underbrace{1,\dots,1}_l)
\end{align}

Until the end of this subsection, let us fix a partition $\lambda$.

The \emph{Young diagram} of shape $\lambda$ is the set
\[
  D(\lambda) := \{ (i,j) \mid 1 \leq i \leq \ell(\lambda) \text{ and } 1 \leq j \leq \lambda_i \}.
\]
As usual, we visualize it by a collection of boxes; e.g.,
\[
  D(4,3,2,2,1) = \ytableausetup{smalltableaux}\ydiagram{4,3,2,2,1}.
\]

Given a partition $\lambda$ and a number $j \in [1,\lambda_1]$, set
\begin{align}\label{eq: col_j}
  \col_j(\lambda) := \sharp \{ i \in [1,\ell(\lambda)] \mid \lambda_i \geq j \}.
\end{align}
For example, if $\lambda = (4,3,2,2,1)$, then
\[
  (\col_1(\lambda),\col_2(\lambda),\col_3(\lambda),\col_4(\lambda)) = (5,4,2,1).
\]

\begin{defi}\label{def: even col}\normalfont
  We say that the partition $\lambda$ has \emph{even columns} if $\col_j(\lambda)$ is even for all $j \in [1,\lambda_1]$.
\end{defi}

Given a partition $\mu \in \mathrm{Par}$, we write $\mu \subseteq \lambda$ to indicate that $D(\mu) \subseteq D(\lambda)$, or equivalently,
\[
  \mu_i \leq \lambda_i\ \text{ for all } 1 \leq i \leq \ell(\lambda).
\]
For each $\mu \in \Par$ with $\mu \subseteq \lambda$, set
\[
  |\lambda/\mu| := |\lambda|-|\mu| \text{ and } D(\lambda/\mu) := D(\lambda) \setminus D(\mu).
\]

\begin{defi}\label{def: vertical strip}\normalfont
  Let $\mu \in \Par$ such that $\mu \subseteq \lambda$.
  We say that $\lambda/\mu$ is a \emph{vertical strip}, and write $\mu \underset{\text{vert}}{\subseteq} \lambda$ if
  \[
    \mu_i \geq \lambda_i-1 \ \text{ for all } i \in [1,\ell(\lambda)].
  \]  
\end{defi}

For example, if $\lambda = (4,3,2,2,1)$ and $\mu = (3,3,1,1)$, then $\mu \underset{\text{vert}}{\subseteq} \lambda$:
\[
  D(\lambda/\mu) = \ydiagram{3+1, 3+0, 1+1, 1+1, 0+1}
\]

The following is immediate from the definition of vertical strips.

\begin{lem}\label{lem: characterization of vertical strips}
  Let $\lambda'$ denote the partition $(\lambda_1-1,\dots,\lambda_{\ell(\lambda)}-1)$, and $\mu$ be a partition such that $\lambda' \underset{\text{vert}}{\subseteq} \mu$.
  Then, we have $\mu \underset{\text{vert}}{\subseteq} \lambda$ if and only if $\ell(\mu) \leq \ell(\lambda)$.
\end{lem}

\subsection{Semistandard tableaux}\label{subsect: sstab}
In this subsection, we fix a partition $\lambda$.

A \emph{Young tableau}, or simply a \emph{tableau}, of shape $\lambda$ is a map
\[
  T: D(\lambda) \rightarrow \mathbb{Z}_{> 0}.
\]
The partition $\lambda$ is called the \emph{shape} of $T$, and denoted by $\sh(T)$.
As usual, we visualize a tableau by filling the boxes of $D(\lambda)$ with positive integers; e.g., the following is a tableau of shape $(4,3,2,2,1)$.
\begin{align}\label{eq: example of a tableau}
  \begin{ytableau}
    1 & 1 & 2 & 4 \\
    2 & 2 & 3 \\
    4 & 4 \\
    5 & 6 \\
    6
  \end{ytableau}.
\end{align}

Let $T$ be a tableau of shape $\lambda$.
Given a positive integer $k$, let $T[k]$ denote the number of occurrences of the entry $k$ in $T$:
\begin{align}\label{eq: T[a]}
  T[k] := \sharp \{ (i,j) \in D(\lambda) \mid T(i,j) = k \}.  
\end{align}

For each $j \in [1,\lambda_1]$, set
\[
  w^\mathrm{col}_j(T) := (T(\col_j(\lambda),j), \dots,T(2,j),T(1, j))
\]
(see \eqref{eq: col_j} for the definition of $\col_j$).
It is the sequence of entries in the $j$-th column of $T$ read from bottom to top.
The \emph{column word} of $T$ is the sequence $w^\mathrm{col}(T)$ of entries obtained by concatenating the $w^\mathrm{col}_j(T)$'s:
\begin{align}\label{eq: w^col(T)}
  w^\mathrm{col}(T) := w^\mathrm{col}_1(T) \circ \cdots \circ w^\mathrm{col}_{\lambda_1}(T).
\end{align}
For example, if $T$ is the tableau in \eqref{eq: example of a tableau}, then
\[
  w^\mathrm{col}(T) = (6,5,4,2,1,6,4,2,1,3,2,4).
\]

The tableau $T$ is said to be \emph{semistandard} if the entries increase weakly from left to right along the rows, and strictly from top to bottom along the columns.
Namely,
\[
  T(i,j) \leq T(i,j+1) \text{ and } T(i,j) < T(i+1,j)\ \text{ for all } (i,j) \in D(\lambda),
\]
where we set $T(i',j') := \infty$ if $(i',j') \notin D(\lambda)$.
For example, the tableau in \eqref{eq: example of a tableau} is semistandard.
Let $\mathrm{SST}_m(\lambda)$ denote the set of all semistandard tableaux of shape $\lambda$ with entries in $[1,m]$.

The generating function
\begin{align}\label{eq: Schur}
  s_\lambda(x_1,\dots,x_m) := \sum_{T \in \mathrm{SST}_m(\lambda)} \mathbf{x}^{\operatorname{wt}(T)} \in \mathbb{Z}[x_1,\dots,x_m]
\end{align}
is called the \emph{Schur function}, where
\begin{align}\label{eq: wt(T)}
  \operatorname{wt}(T) := (T[1], \dots, T[m]), \quad \mathbf{x}^{(a_1,\dots,a_m)} := x_1^{a_1} \cdots x_m^{a_m}.
\end{align}
The Schur functions are symmetric polynomials:
\[
  s_\lambda(x_{\sigma(1)}, \dots, x_{\sigma(m)}) = s_\lambda(x_1,\dots,x_m)
\]
for all permutation $\sigma$ on $[1,m]$.

For $\mu \in \Par$ with $\mu \subseteq \lambda$, a \emph{tableau} of shape $\lambda/\mu$ is a map
\[
  D(\lambda/\mu) \rightarrow \mathbb{Z}_{> 0}.
\]
Let $\Tab(\lambda/\mu)$ denote the set of all tableaux of shape $\lambda/\mu$.
The notion of semistandard tableaux of shape $\lambda/\mu$ is defined in the obvious way.
Let $\SST_m(\lambda/\mu)$ denote the set of all semistandard tableaux of shape $\lambda/\mu$ with entries in $[1,m]$.
For example, the following is a semistandard tableau of shape $(4,3,2,2,1)/(2,2,1)$:
\[
  \begin{ytableau}
    \none & \none & 2 & 2 \\
    \none & \none & 3 \\
    \none & 1 \\
    2 & 4 \\
    6
  \end{ytableau}.
\]

\subsection{Plactic monoid}
The set of all semistandard tableaux with entries in $[1,m]$ forms a monoid, called the \emph{plactic monoid} ({\it cf.}\ \cite[Sections 1.1 and A.2]{Ful97}).
For the reader's convenience, we recall here its definition.
In order to describe the monoid structure, we need to introduce the \emph{column insertion algorithm}, which receives a pair $(w, T)$ of a positive integer $w$ and a semistandard tableau $T$ as an input, and returns a new semistandard tableau $w \rightarrow T$ as an output as follows.
Set $\lambda := \sh(T)$, and
\[
  w_0 := w, \quad T_0 := T.
\]
For each $j > 0$, given a pair $(w_{j-1}, T_{j-1})$, set
\begin{align}\label{eq: col ins}
  \begin{split}
    &r_{j} := \min \{ r \in [1,\col_{j}(\lambda)+1] \mid T(r,j) \geq w_{j-1} \}, \\
    &w_j := T(r_j, j),
  \end{split}
\end{align}
where we set $T(i',j') = \infty$ if $(i',j') \notin D(\lambda)$ (see \eqref{eq: col_j} for the definition of $\col_j$).
Also set $\lambda^j$ to be the partition such that
\[
  D(\lambda^j) = \begin{cases}
    D(\lambda) & \text{ if } r_j \leq \col_j(\lambda), \\
    D(\lambda) \sqcup \{ (r_j, j) \} & \text{ if } r_j = \col_j(\lambda)+1,
  \end{cases}
\]
and $T_j$ to be the semistandard tableau of shape $\lambda^j$ such that
\[
  T_j(i',j') = \begin{cases}
    w_{j-1} & \text{ if } (i',j') = (r_j,j), \\
    T_{j-1}(i',j') & \text{ if } (i',j') \neq (r_j,j),
  \end{cases} \ \text{ for all } (i',j') \in D(\lambda^j).
\]
Let $s \geq 1$ denote the minimal integer such that $r_j = \col_j(\lambda)+1$.
Then, the semistandard tableau $T_s$ is the one $w \rightarrow T$.

The sequence
\[
  \br(w, T) := (r_1, \dots, r_s)
\]
is called the \emph{bumping route}.

\begin{ex}\normalfont
  Let $w = 7$ and
  \[
    T = \ytableausetup{smalltableaux}
    \begin{ytableau}
      1 & 2 & 2 & 3 \\
      3 & 4 & 5 \\
      4 & 5 & 6 \\
      6 & 6 & 9 \\
      7 & 7 & 10 \\
      8 & 8 \\
      10
    \end{ytableau}.
  \]
  Then, we have
  \[
    w \rightarrow T = 
    \begin{ytableau}
      1 & 2 & 2 & 3 \\
      3 & 4 & 5 & 9 \\
      4 & 5 & 6 \\
      6 & 6 & 7 \\
      7 & 7 & 10 \\
      8 & 8 \\
      10
    \end{ytableau}, \quad \br(w, T) = (5,5,4,2).
  \]
\end{ex}

The following proposition can be straightforwardly deduced from the definitions.

\begin{prop}\label{prop: properties of colum insertion}
  Let $w, T, r_1,\dots,r_s$ be as above.
  Set $S := w \rightarrow T$ and $\mu := \sh(S)$.
  Then, the following hold.
  \begin{enumerate}
    \item\label{item: prop properties of colum insertion 1} $r_1 \geq \cdots \geq r_s$
    \item\label{item: prop properties of colum insertion 2} $w \leq T(r_1,1) \leq T(r_2,2) \leq \dots \leq T(r_{s-1},s-1) < T(r_s,s) = \infty$.
    \item\label{item: prop properties of colum insertion 3} $\lambda \subset \mu$. Moreover, $D(\mu/\lambda) = \{ (r_s, s) \}$.
    \item\label{item: prop properties of colum insertion 4} For each $(i,j) \in D(\mu)$, we have
    \[
      S(i,j) = \begin{cases}
        w & \text{ if } (i,j) = (r_1,1), \\
        T(r_{j-1}, j-1) & \text{ if } j \in [2,s] \text{ and } i = r_j, \\
        T(i,j) & \text{ otherwise}.
      \end{cases}
    \]
  \end{enumerate}
\end{prop}

Given two semistandard tableaux $S$ and $T$, their product $S*T$ in the plactic monoid is given by
\[
  S*T := w_1 \rightarrow (\cdots \rightarrow (w_r \rightarrow T) \cdots )
\]
where $(w_1, \dots, w_r) = w^\mathrm{col}(S)$ (see \eqref{eq: w^col(T)} for the definition of $w^\col$).

Regarding bumping routes of successive insertions, the following is known.

\begin{prop}[{\cite[Exercise 3 in Section A.2]{Ful97}}]\label{prop: col bump lem}
  Let $\lambda \in \Par_{\leq m}$, $T \in \SST_m(\lambda)$, and $w,w' \in [1,m]$ be such that $w < w'$.
  Let us write
  \[
    \br(w, T) = (r_1,\dots,r_s), \quad \br(w', (w \rightarrow T)) = (r'_1,\dots,r'_{s'}).
  \]
  Then, we have $s' \leq s$ and $r'_j > r_j$ for all $j \in [1,s']$.
\end{prop}

The following is known as (a combinatorial version of) Pieri's formula.

\begin{prop}\label{prop: Pieri formula}
  Let $\lambda \in \Par_{\leq m}$ and $k \in [0,m]$.
  The assignment $(S,T) \rightarrow S*T$ gives rise to a bijection
  \[
    *: \SST_m(\varpi_k) \times \SST_m(\lambda) \rightarrow \bigsqcup_{\substack{\mu \in \Par_{\leq m} \\ \lambda \underset{\text{vert}}{\subseteq} \mu \text{ and } |\mu/\lambda| = k}} \SST_m(\mu).
  \]
\end{prop}

\subsection{Symplectic tableaux}
In this subsection, we fix a partition $\lambda \in \Par_{\leq 2n}$.

\begin{defi}[{\cite[Section 4]{Ki76}}]\label{def: SpT}\normalfont
  A semistandard tableau $T \in \mathrm{SST}_{2n}(\lambda)$ is said to be \emph{symplectic} if
  \[
    T(k,1) \geq 2k-1 \ \text{ for all } k \in [1,\ell(\lambda)].
  \]
  Let $\SpT_{2n}(\lambda)$ denote the set of all symplectic tableaux of shape $\lambda$.
\end{defi}

\begin{prop}
  If $\SpT_{2n}(\lambda) \neq \emptyset$, then $\ell(\lambda) \leq n$.
\end{prop}

\begin{proof}
  Assume contrary that $\ell(\lambda) > n$.
  Let $T \in \SpT_{2n}(\lambda)$.
  Then, we have
  \[
    T(n+1, 1) \geq 2(n+1) - 1 = 2n+1 > 2n.
  \]
  This contradicts that the entries of $T$ are in $[1,2n]$.
  Thus, the assertion follows.
\end{proof}

For each $\nu \in \mathrm{Par}_{\leq n}$, the generating function
\begin{align}\label{eq: sp Schur}
  s^{Sp}_\nu(y_1,\dots,y_n) := \sum_{T \in \mathrm{SpT}_{2n}(\nu)} \mathbf{y}^{\operatorname{wt}^{Sp}(T)} \in \mathbb{Z}[y_1^{\pm 1},\dots,y_n^{\pm 1}]
\end{align}
is called the \emph{symplectic Schur function}, where
\begin{align}\label{eq: wtsp}
  \operatorname{wt}^{Sp}(T) := (T[1]-T[2], T[3]-T[4], \dots, T[2n-1]-T[2n]).
\end{align}
The symplectic Schur functions are linearly independent.

\begin{lem}\label{lem: existence of removable entries}
  Let $T \in \mathrm{SST}_{2n}(\lambda)$.
  If $T$ is not symplectic, then there exists a unique $i \in [2,2n]$ such that
  \[
    T(i,1) < 2i-1 \text{ and } T(k,1) \geq 2k-1 \text{ for all } k \in [1,i-1].
  \]
  Moreover, we have
  \[
    T(i,1) = 2i-2 \text{ and } T(i-1,1) = 2i-3.
  \]
\end{lem}

\begin{proof}
  Since $T$ is not symplectic, there exists $i \in [1,\ell(\lambda)]$ such that $T(i,1) < 2i-1$ (note that the number $i$ cannot be $1$ since $T(1,1)$ is always greater than or equal to $1\ (= 2 \cdot 1-1)$).
  We may take the minimal $i$ among such integers.
  Then, the first assertion is clear.

  Now, we have
  \[
    2i-3 = 2(i-1)-1 \leq T(i-1,1) < T(i,1) < 2i-1.
  \]
  This implies the second assertion.
\end{proof}

\subsection{Reduction map}
For each $l \in [0,m]$, a tableau in $\mathrm{SST}_m(\varpi_l)$ (see \eqref{eq: varpi_l} for the definition of $\varpi_l$) can be represented by an increasing sequence $(a_1,\dots,a_l)$ of integers in $[1,m]$.
We often regard such sequences $\mathbf{a}$, $\mathbf{b}$, etc.\ as sets, and consider their cardinalities $|\mathbf{a}|$, disjoint unions $\mathbf{a} \sqcup \mathbf{b}$, set differences $\mathbf{a} \setminus \mathbf{b}$, and so on.

In this subsection, we fix $l \in [0,2n]$ and $\mathbf{a} = (a_1,\dots,a_l) \in \SST_{2n}(\varpi_l)$.

\begin{defi}\label{def: rem}\normalfont
  The set of \emph{removable entries} of $\mathbf{a}$ is the subset $\Rem(\mathbf{a})$ of $\mathbf{a}$ defined as follows.
  \begin{enumerate}
    \item\label{item: rem 1} If $l \leq 1$, then $\Rem(\mathbf{a}) = \emptyset$.
    \item\label{item: rem 2} If $l > 1$, then
    \[
      \Rem(\mathbf{a}) = \begin{cases}
        \Rem(a_1,\dots,a_{l-2}) \sqcup \{ a_{l-1}, a_l \} & \text{ if } a_l \in 2\mathbb{Z},\ a_{l-1} = a_l-1, \text{ and } \\
        & \text{ } a_l < 2l - |\Rem(a_1,\dots,a_{l-2})| - 1, \\
        \Rem(a_1 ,\dots,a_{l-1}) & \text{ otherwise}.
      \end{cases}
    \]
  \end{enumerate}
\end{defi}

\begin{defi}\label{def: red}\normalfont
  The \emph{reduction map} on $\SST_{2n}(\varpi_l)$ is the map
  \[
    \red: \SST_{2n}(\varpi_l) \rightarrow \bigsqcup_{k = 0}^{2n} \SST_{2n}(\varpi_k);\ \mathbf{a} \mapsto \mathbf{a} \setminus \Rem(\mathbf{a}).
  \]
\end{defi}

For example, if we take $\mathbf{a} = (1,3,4,5,6,7,11,12,13,14)$, then we have
  \[
    \Rem(\mathbf{a}) = \{ 3,4,5,6,13,14 \}, \quad \red(\mathbf{a}) = (1,7,11,12).
  \]

The following is immediate from the definition.

\begin{prop}\label{prop: properties of rem}\hfill
  \begin{enumerate}
    \item\label{item: properties of rem 1} If $a_l \in \Rem(\mathbf{a})$, then $a_l \in 2\mathbb{Z}$.
    \item\label{item: properties of rem 2} If $a_l \notin \Rem(\mathbf{a})$, then $\Rem(\mathbf{a}) = \Rem(a_1,\dots,a_{l-1})$.
    \item\label{item: properties of rem 2.1} If $a_l \in \Rem(\mathbf{a})$, then $\Rem(a_1,\dots,a_{l-1}) = \Rem(a_1,\dots,a_{l-2})$.
    \item\label{item: properties of rem 3} $\Rem(a_1,\dots,a_k) \subseteq \Rem(\mathbf{a})$ for all $k \in [0,l]$.
  \end{enumerate}
\end{prop}

\subsection{Successor map}
In this subsection, we fix $\lambda \in \Par_{\leq 2n}$, and set $l := \ell(\lambda)$.
Let $\lambda'$ denote the partition $(\lambda_1-1,\dots,\lambda_l-1)$.

Let us define a map
\begin{align}\label{eq: map d}
  d: \SST_{2n}(\lambda) \rightarrow \SST_{2n}(\varpi_l) \times \SST_{2n}(\lambda')
\end{align}
as follows.
For each $T \in \SST_{2n}(\lambda)$, the image $d(T)$ is the pair $(\mathbf{a}, T')$ consisting of the first column $\mathbf{a}$ of $T$:
\[
  \mathbf{a} = (T(1,1), T(2,1), \dots, T(l,1)),
\]
and the other part $T'$:
\[
  T'(i,j) = T(i,j+1) \ \text{ for all } (i,j) \in D(\lambda').
\]

\begin{defi}\normalfont
  The \emph{successor map} is the composite
  \[
    \suc := * \circ (\red, \mathrm{id}) \circ d: \SST_{2n}(\lambda) \rightarrow \bigsqcup_{\mu \in \Par_{\leq 2n}} \SST_{2n}(\mu),
  \]
  where $*$ denotes the multiplication map of the plactic monoid.
\end{defi}

\begin{ex}\normalfont
  Let
  \[
    T := 
    \begin{ytableau}
      1 & 2 & 2 & 3 \\
      3 & 3 & 4 & 5 \\
      4 & 4 & 5 & 6 \\
      5 & 6 & 6 & 9 \\
      6 & 7 & 7 & 10 \\
      7 & 8 & 8 \\
      11 & 10 \\
      12 \\
      13 \\
      14
    \end{ytableau}.
  \]
  Then, we have
  \begin{align*}
    \begin{split}
      &\suc(T) = 
      \begin{ytableau}
        1 \\
        7 \\
        11 \\
        12
      \end{ytableau} *
      \begin{ytableau}
        2 & 2 & 3 \\
        3 & 4 & 5 \\
        4 & 5 & 6 \\
        6 & 6 & 9 \\
        7 & 7 & 10 \\
        8 & 8 \\
        10 \\
      \end{ytableau} = 
      \begin{ytableau}
        1 & 2 & 2 & 3 \\
        3 & 4 & 5 & 9 \\
        4 & 5 & 6 \\
        6 & 6 & 7 \\
        7 & 7 & 10 \\
        8 & 8 \\
        10 \\
        11 \\
        12
      \end{ytableau}, \quad \suc^2(T) = 
      \begin{ytableau}
        1 \\
        6 \\
        10
      \end{ytableau} * 
      \begin{ytableau}
        2 & 2 & 3 \\
        4 & 5 & 9 \\
        5 & 6 \\
        6 & 7 \\
        7 & 10 \\
        8 \\
      \end{ytableau} = 
      \begin{ytableau}
        1 & 2 & 2 & 3 \\
        4 & 5 & 6 & 9 \\
        5 & 6 \\
        6 & 7 \\
        7 & 10 \\
        8 \\
        10
      \end{ytableau}, \\
      &\suc^3(T) = 
      \begin{ytableau}
        1 \\
        4 \\
        10
      \end{ytableau} * 
      \begin{ytableau}
        2 & 2 & 3 \\
        5 & 6 & 9 \\
        6 \\
        7 \\
        10
      \end{ytableau} = 
      \begin{ytableau}
        1 & 2 & 2 & 3 \\
        4 & 5 & 6 & 9 \\
        6 & 10 \\
        7 \\
        10
      \end{ytableau}, \quad \suc^4(T) = 
      \begin{ytableau}
        1 \\
        4 \\
        6 \\
        7 \\
        10
      \end{ytableau} * 
      \begin{ytableau}
        2 & 2 & 3 \\
        5 & 6 & 9 \\
        10
      \end{ytableau} = 
      \begin{ytableau}
        1 & 2 & 2 & 3 \\
        4 & 5 & 6 & 9 \\
        6 & 10 \\
        7 \\
        10
      \end{ytableau}
    \end{split}
  \end{align*}
\end{ex}

\begin{lem}\label{lem: properties of suc(T)}
  Let $T \in  \SST_{2n}(\lambda)$.
  Set $d(T) = (\mathbf{a}, T')$, and write $\mathbf{a} = (a_1,\dots,a_l)$ and $\red(\mathbf{a}) = \mathbf{b} = (b_1,\dots,b_k)$.
  Set $S^0 := T'$, $S^t := b_t \rightarrow S^{t-1}$, and $\mu^t := \sh(S^t)$.
  Let us write $\mathrm{br}(b_t, S^{t-1}) = (r_{t,1},\dots,r_{t,s_t})$.
  Let $r_{t,0}$ be such that $a_{r_{t,0}} = b_t$.
  Then, the following hold for all $t \in [1,k]$:
  \begin{enumerate}
    \item\label{item: properties of suc(T) 1} $D(\mu^t/\lambda') = \{ (r_{u,s_u}, s_u) \mid u \in [1,t] \}$.
    \item\label{item: properties of suc(T) 2} $S^t(i,j) = \begin{cases}
      T(r_{u,j-1},j) & \text{ if } u \in [1,t],\ j \in [1,s_u], \text{ and } i = r_{u,j}, \\
      T(i,j+1) & \text{ otherwise}.
    \end{cases}$
    \item\label{item: properties of suc(T) 2.1} $s_1 \geq s_2 \geq \dots \geq s_k$.
    \item\label{item: properties of suc(T) 2.2} $r_{t,j} > r_{t-1,j} > \dots > r_{1,j}$ for all $j \in [0,s_t]$.
    \item\label{item: properties of suc(T) 3} $r_{t,j} = \min\{ r \in [r_{t-1,j}+1, \col_j(\mu^{t-1})+1] \mid T(r,j+1) \geq T(r_{t,j-1}, j) \}$ for all $j \in [1,s_t]$, where we set $r_{0,j} = 0$.
    \item\label{item: properties of suc(T) 4} $r_{t,0} \geq r_{t,1} \geq \dots \geq r_{t,s_t}$.
  \end{enumerate}
\end{lem}

\begin{proof}
  The first and second assertions can be deduced from Proposition \ref{prop: properties of colum insertion} \eqref{item: prop properties of colum insertion 3} and \eqref{item: prop properties of colum insertion 4} by induction on $t$.
  The third and fourth assertions follow from Proposition \ref{prop: col bump lem}.

  Next, let us prove the fifth assertion.
  For each $j \in [0,s_t]$, set
  \[
    w_{t,j} := \begin{cases}
      b_t & \text{ if } j = 0, \\
      S^{t-1}(r_{t,j},j) & \text{ if } j > 0.
    \end{cases}
  \]
  Then, by equation \eqref{eq: col ins}, we have
  \begin{align}\label{eq: 6}
    r_{t,j} = \min\{ r \in [1,\col_j(\mu^{t-1})+1] \mid S^{t-1}(r,j) \geq w_{t,j-1} \}.
  \end{align}
  This, together with the second and fourth assertions, implies our claim.

  Finally, let us prove the last assertion by induction on $t$; we understand that $s_0 = 0$ so that the assertion for $t = 0$ is clear.
  We have
  \[
    r_{t,1} \geq \dots \geq r_{t,s_t}
  \]
  by Proposition \ref{prop: properties of colum insertion} \eqref{item: prop properties of colum insertion 1}.
  Hence, we only need to show that $r_{t,0} \geq r_{t,1}$.
  By equation \eqref{eq: 6}, we have
  \[
    r_{t,1} \leq \col_1(\mu^{t-1})+1.
  \]
  Hence, there is nothing to prove if $r_{t,0} > \col_1(\mu^{t-1})$.
  Therefore, suppose that $r_{t,0} \leq \col_1(\mu^{t-1})$.
  Then, we have
  \[
    T(r_{t,0},1) \leq T(r_{t,0},2).
  \]
  Since $a_{r_{t,0}} = b_t > b_{t-1} = a_{r_{t-1,0}}$, we see that $r_{t,0} > r_{t-1,0}$.
  By our induction hypothesis, it holds that $r_{t-1,0} \geq r_{t-1,1}$.
  Summarizing above, we have
  \[
    r_{t-1,1} < r_{t,0} \leq \col_1(\mu^{t-1}) \text{ and } T(r_{t,0}, 2) \geq T(r_{t,0}, 1).
  \]
  Then, the fifth assertion implies that $r_{t,1} \leq r_{t,0}$, as desired.
\end{proof}

\begin{lem}\label{lem: stationality of suc}
  Let $T \in \SST_{2n}(\lambda)$ and set $\mu := \sh(\suc(T))$.
  Then, we have $\mu \underset{\text{vert}}{\subseteq} \lambda$.
  Moreover, the equality holds if and only if $\suc(T) = T$.
\end{lem}

\begin{proof}
  We use the notation in Lemma \ref{lem: properties of suc(T)}.
  Note that $S^k = \suc(T)$ and $\mu^k = \mu$.
  By Proposition \ref{prop: Pieri formula}, we see that
  \[
    \lambda' \underset{\text{vert}}{\subseteq} \mu \text{ and } |\mu/\lambda'| = |\red(\mathbf{a})|.
  \]
  Therefore, in order to prove that $\mu \underset{\text{vert}}{\subseteq} \lambda$, we only need to show that $\ell(\mu) \leq l$ (see Lemma \ref{lem: characterization of vertical strips}).

  For each $t \in [0,k]$, by Lemma \ref{lem: properties of suc(T)} \eqref{item: properties of suc(T) 1} we have
  \begin{align}\label{eq: 5}
    \ell(\mu) = \ell(\lambda') + \sharp\{ t \in [1,k] \mid s_t = 1 \}.
  \end{align}
  Let $u \in [1,k]$ be such that $r_{u,0} \leq \ell(\lambda')$.
  Then, we have
  \[
    r_{u,1} \leq r_{u,0} \leq \ell(\lambda') \leq \ell(\mu^{u-1}) = \col_1(\mu^{u-1}),
  \]
  where the first inequality follows from Lemma \ref{lem: properties of suc(T)} \eqref{item: properties of suc(T) 4}.
  This implies that $s_u > 1$.
  Hence, we see that
  \[
    \sharp\{ t \in [1,k] \mid s_u = 1 \} \leq k-\ell(\lambda').
  \]
  This, together with equation \eqref{eq: 5}, implies the first assertion.
  
  Next, let us prove the second assertion.
  Suppose that $\mu = \lambda$.
  In this case, we must have $\red(\mathbf{a}) = \mathbf{a}$ since
  \[
    0 = |\lambda/\mu| = |\lambda/\lambda'| - |\mu/\lambda'| = l - |\red(\mathbf{a})| = |\Rem(\mathbf{a})|.
  \]
  On the other hand, it is clear that
  \[
    \mathbf{a} * T' = T.
  \]
  Therefore, we obtain $\suc(T) = T$, as desired.
  The opposite direction is trivial.
\end{proof}

\section{Main result}\label{sect: main}

\subsection{Statement}
Let $\lambda \in \Par_{\leq 2n}$ and $T \in \SST_{2n}(\lambda)$.
For each $k \geq 0$, define a semistandard tableau $P^k$, a partition $\nu^k$, and a tableau $Q^k$ inductively as follows.
First, set
\[
  P^0 := T, \quad \nu^0 := \lambda,
\]
and $Q^0$ to be the unique tableau of shape $\lambda/\lambda$.
For $k \geq 0$, set
\[
  P^{k+1} := \suc(P^k), \quad \nu^{k+1} := \sh(P^{k+1}),
\]
and $Q^{k+1}$ to be the tableau of shape $\lambda/\nu^{k+1}$ given by
\[
  Q^{k+1}(i,j) := \begin{cases}
    Q^k(i,j) & \text{ if } (i,j) \notin D(\nu^k), \\
    k+1 & \text{ if } (i,j) \in D(\nu^k).
  \end{cases}
\]
Note that by Lemma \ref{lem: stationality of suc}, there exists a unique $k_0 \geq 0$ such that
\[
  \lambda = \nu^0 \underset{\text{vert}}{\supset} \nu^1 \underset{\text{vert}}{\supset} \nu^2 \underset{\text{vert}}{\supset} \cdots \underset{\text{vert}}{\supset} \nu^{k_0}
\]
and
\[
  P^k = P^{k_0}, \ \nu^k = \nu^{k_0}, \ Q^k = Q^{k_0} \ \text{ for all } k \geq k_0.
\]

\begin{defi}\label{def: PAII and QAII}\normalfont
  Let $T, k_0$ be as above.
  Define two tableaux $\PAII(T)$ and $\QAII(T)$ to be $P^{k_0}$ and $Q^{k_0}$, respectively.
\end{defi}

See Example \ref{ex: LR} for an example of this algorithm.
The following is immediate from the definition.

\begin{lem}\label{lem: QAII, nu}
  Let $\lambda,T,\nu^k,k_0$ be as above.
  Set $\nu := \nu^{k_0}$.
  \begin{enumerate}
    \item For each $(i,j) \in D(\lambda/\nu)$, we have
    \[
      \QAII(T)(i,j) = \min\{ k \in [1,k_0] \mid (i,j) \notin D(\nu^k) \}.
    \]
    \item For each $k \geq 0$, we have
    \[
      D(\nu^k) = D(\nu) \sqcup \{ (i,j) \in D(\lambda/\nu) \mid \QAII(T)(i,j) > k \}.
    \]
  \end{enumerate}
\end{lem}

\begin{defi}\normalfont
  Let $\lambda \in \Par_{\leq 2n}$ and $\nu \in \Par_{\leq n}$ be such that $\nu \subseteq \lambda$.
  A tableau $Q$ of shape $\lambda/\nu$ is said to be a \emph{recording tableau} if there exists $T \in \SST_{2n}(\lambda)$ such that $\QAII(T) = Q$.
  Let $\Rec_{2n}(\lambda/\nu)$ denote the set of recording tableaux of shape $\lambda/\nu$.
\end{defi}

Let $\lambda \in \Par_{\leq 2n}$ and $\nu \in \Par_{\leq n}$ be such that $\nu \subseteq \lambda$.
Let $\widetilde{\Rec}_{2n}(\lambda/\nu)$ denote the set of tableaux $Q$ of shape $\lambda/\nu$ satisfying the following:
\begin{enumerate}
  \item[(R1)]\label{item: main thm a} The entries of $Q$ strictly decrease along the rows from left to right.
  \item[(R2)]\label{item: main thm b} The entries of $Q$ weakly decrease along the columns from top to bottom.
  \item[(R3)]\label{item: main thm c} For each $k > 0$, the number $Q[k]$ $($see \eqref{eq: T[a]} for the definition$)$ is even.
  \item[(R4)]\label{item: main thm d} For each $k > 0$, it holds that
  \[
    Q[k] \geq 2(\ell(\nu^{k-1})-n),
  \]
  where $\nu^{k-1}$ is the partition such that
  \[
    D(\nu^{k-1}) = D(\nu) \sqcup \{ (i,j) \in D(\lambda/\nu) \mid Q(i,j) \geq k \}.
  \]
  \item[(R5)]\label{item: main thm e} For each $r,k > 0$, let $Q_{\leq r}[k]$ denote the number of occurrences of $k$ in $Q$ in the $r$-th row or above.
  Then, the following inequality holds:
  \[
    Q_{\leq r}[k+1] \leq Q_{\leq r}[k].
  \]
\end{enumerate}

Now, we are ready to state the main result in this paper.

\begin{thm}\label{thm: main}
  Let $\lambda \in \Par_{\leq 2n}$.
  \begin{enumerate}
    \item\label{item: main thm 1} The assignment $T \mapsto (\PAII(T), \QAII(T))$ gives rise to a bijection
    \[
      \LRAII : \SST_{2n}(\lambda) \rightarrow \bigsqcup_{\substack{\nu \in \Par_{\leq n} \\ \nu \subseteq \lambda}} (\SpT_{2n}(\nu) \times \Rec_{2n}(\lambda/\nu)).
    \]
    \item\label{item: main thm 2} For each $\nu \in \Par_{\leq n}$ such that $\nu \subseteq \lambda$, we have
    \[
      \Rec_{2n}(\lambda/\nu) = \widetilde{\Rec}_{2n}(\lambda/\nu).
    \]
  \end{enumerate}
\end{thm}

\begin{defi}\normalfont
  We call the map $\LRAII$ in Theorem \ref{thm: main} the \emph{Littlewood-Richardson map}.
\end{defi}

The rest of this paper is devoted to proving the theorem.
Since the proof is involved, we give an outline in the next subsection.

\subsection{Outline of the proof}
First, we reformulate the reduction map as a composite of certain maps (Corollary \ref{cor: factorization of red} \eqref{item: factorization of red 1}).
As a result, we see that the reduction map on $\SST_{2n}(\varpi_l)$ is injective (Corollary \ref{cor: factorization of red} \eqref{item: factorization of red 2}).

Next, by studying some properties of the reduction map, we deduce that the tableau $\PAII(T)$ is symplectic (Proposition \ref{prop: codomain of LR}).
Also, we prove the injectivity of the Littlewood-Richardson map by using the successor map (Proposition \ref{prop: inj of LR}).

To prove the surjectivity of the Littlewood-Richardson map, we use representation theory of a quantum symmetric pair $(\mathbf{U}, \mathbf{U}^\imath)$ of type $A\mathrm{II}_{2n-1}$.
It is known that for each $\lambda \in \Par_{\leq 2n}$, there exists an irreducible $\mathbf{U}$-module $V(\lambda)$ with basis of the form $\{ b_T \mid T \in \SST_{2n}(\lambda) \}$.
Similarly, for each $\nu \in \Par_{\leq n}$, there exists an irreducible $\mathbf{U}^\imath$-module $V^\imath(\nu)$.
We show that it admits a distinguished basis of the form $\{ b^\imath_T \mid T \in \SpT_{2n}(\nu) \}$ (Proposition \ref{prop: p_nu}).
Then, we lift the Littlewood-Richardson map to a $\mathbf{U}^\imath$-module homomorphism
\[
  V(\lambda) \rightarrow \bigoplus_{\substack{\nu \in \Par_{\leq n} \\ \nu \subseteq \lambda}} (V^\imath(\nu) \otimes \mathbb{Q}(q) \Rec_{2n}(\lambda/\nu))
\]
which maps $b_T$ to $b_{\PAII(T)} \otimes \QAII(T)$ modulo $q^{-1}$ for all $T \in \SST_{2n}(\lambda)$ (Proposition \ref{prop: quantum LR is comb LR at q infty}).
Here, each $x \in \mathbf{U}^\imath$ acts on each summand of the right-hand side as $x \otimes \mathrm{id}$.
The complete reducibility of the right-hand side implies that this homomorphism is surjective (Theorem \ref{thm: quantum LR is isom}).
Hence, we conclude that the Littlewood-Richardson map is surjective (Corollary \ref{cor: surj of LR}).

Finally, by studying the successor map in detail, we verify that $\Rec_{2n}(\lambda/\nu) \subseteq \widetilde{\Rec}_{2n}(\lambda/\nu)$ (Lemma \ref{lem: rec in rectil}).
On the other hand, we show that there is an injective map from $\widetilde{\Rec}_{2n}(\lambda/\nu)$ to the set $\LRSp_{2n}(\lambda/\nu)$ of symplectic Littlewood-Richardson tableaux (Lemma \ref{lem: rectil in lrt}).
Then, we conclude that $\Rec_{2n}(\lambda/\nu) = \widetilde{\Rec}_{2n}(\lambda/\nu)$ (Theorem \ref{thm: rec = rectil}) by proving that both $|\Rec_{2n}(\lambda/\nu)|$ and $|\LRSp_{2n}(\lambda/\nu)|$ are equal to the multiplicity $m_{\lambda,\nu}$ by Sundaram's branching rule (Theorem \ref{thm: Sun90}) and the bijectivity of the Littlewood-Richardson map.

\section{Factorization of the reduction map}\label{sect: factor red}
The aim of this section is to prove Corollary \ref{cor: factorization of red}, which describes the reduction map as a composite of several maps.
The other part is devoted to preparing the proof.

\subsection{Combinatorial $R$-matrices}
Recall from \eqref{eq: varpi_l} the partitions $\varpi_l$ for $l \geq 0$.

\begin{defi}[{{\it cf}.\ \cite[Rule 3.10]{NY97}}]\label{def: comb R}\normalfont
  Let $k,l \in [0, m]$.
  The \emph{combinatorial $R$-matrix} is the map
  \[
    R = R_{k,l} : \SST_m(\varpi_k) \times \SST_m(\varpi_l) \rightarrow \SST_m(\varpi_l) \times \SST_m(\varpi_k)
  \]
  defined as follows.
  Let $\mathbf{a} = (a_1,\dots,a_k) \in \SST_m(\varpi_k)$ and $\mathbf{b} = (b_1,\dots,b_l) \in \SST_m(\varpi_l)$.
  \begin{enumerate}
    \item\label{item: def comb R 1} When $k \leq l$.
    For each $r \in [1,k]$, define $i_r \in [1,l]$ inductively as follows.
    Set $i_1$ to be the minimum $i \in [1,l]$ such that $b_i \geq a_1$; when such $i$ does not exist, we set $i_1 := 1$.
    Suppose that $r \geq 2$ and we have determined $i_1,\dots,i_{r-1}$.
    Set $i_r$ to be the minimum $i \in [1,l] \setminus \{ i_1,\dots,i_{r-1} \}$ such that $b_i \geq a_r$; when such $i$ does not exist, we set $i_r := \min([1,l] \setminus \{ i_1,\dots,i_{r-1} \})$.
    Then, we set
    \[
      R(\mathbf{a}, \mathbf{b}) := (\mathbf{a} \sqcup \mathbf{b}'', \mathbf{b}'),
    \]
    where
    \[
      \mathbf{b}' := (b_{i_1}, \dots, b_{i_k}), \quad \mathbf{b}'' := \mathbf{b} \setminus \mathbf{b}'.
    \]

    \item\label{item: def comb R 2} When $k \geq l$.
    For each $r \in [1,l]$, define $i_r \in [1,k]$ inductively as follows.
    Set $i_1$ to be the maximum $i \in [1,k]$ such that $a_i \leq b_1$; when such $i$ does not exist, we set $i_1 := k$.
    Suppose that $r \geq 2$ and we have determined $i_1,\dots,i_{r-1}$.
    Set $i_r$ to be the maximum $i \in [1,k] \setminus \{ i_1,\dots,i_{r-1} \}$ such that $a_i \leq b_r$; when such $i$ does not exist, we set $i_r := \max([1,k] \setminus \{ i_1,\dots,i_{r-1} \})$.
    Then, we set
    \[
      R(\mathbf{a}, \mathbf{b}) := (\mathbf{a}', \mathbf{b} \sqcup \mathbf{a}''),
    \]
    where
    \[
      \mathbf{a}' := (a_{i_1}, \dots, a_{i_l}), \quad \mathbf{a}'' := \mathbf{a} \setminus \mathbf{a}'.
    \]
  \end{enumerate}
\end{defi}

\begin{prop}[{\cite[Proposition 3.21]{NY97}}]\label{prop: Rcomb is bij}
  The map $R_{k,l}$ is a bijection with inverse $R_{l,k}$.
\end{prop}

Given an integer $a \in [1,m]$, set
\begin{align}\label{eq: a^vee}
  a^\vee := [1,m] \setminus \{ a \} \in \SST_m(\varpi_{m-1}).
\end{align}

\begin{lem}\label{lem: Rcomb m mvee}
  We have
  \[
    R(m,m^\vee) = (1^\vee,1).
  \]
\end{lem}

\begin{proof}
  With the same notation as Definition \ref{def: comb R} \eqref{item: def comb R 1}, we see that
  \[
    i_1 = 1.
  \]
  Hence, the assertion follows.
\end{proof}

\begin{lem}\label{lem: Rcomb tvee 1,t}
  Let $t \in [1,m-1]$.
  Then, we have
  \[
    R(t^\vee, (1,\dots,t)) = ((1,\dots,t-1,m), m^\vee).
  \]
\end{lem}

\begin{proof}
  With the same notation as Definition \ref{def: comb R} \eqref{item: def comb R 2}, we see inductively that
  \[
    i_1 = 1, i_2 = 2, \dots, i_{t-1} = t-1, \text{ and } i_t = m-1.
  \]
  Hence, the assertion follows.
\end{proof}

\begin{lem}\label{lem: Rcomb a sub b}
  Let $\mathbf{a} = (a_1,\dots,a_k) \in \SST_m(\varpi_k)$ and $\mathbf{b} = (b_1,\dots,b_l) \in \SST_m(\varpi_l)$ with $k \leq l$.
  Suppose that for each $r \in [1,k]$, there exists $j_r \in [1,l]$ such that $a_r = b_{j_r}$.
  Then, we have
  \[
    R(\mathbf{a}, \mathbf{b}) = (\mathbf{b}, \mathbf{a})
  \]
  and
  \[
    R(\mathbf{b}, \mathbf{a}) = (\mathbf{a}, \mathbf{b}).
  \]
\end{lem}

\begin{proof}
  The first assertion is clear from Definition \ref{def: comb R} \eqref{item: def comb R 1}.
  The second assertion follows from the first one and Proposition \ref{prop: Rcomb is bij}.
\end{proof}

\begin{lem}\label{lem: Rcomb svee a r,s}
  Let $\mathbf{a} = (a_1,\dots,a_k) \in \SST_m(\varpi_k)$, and $r,s \in [1,m]$ be such that $a_k+1 < r \leq s$.
  Then, we have
  \[
    R(s^\vee, \mathbf{a} \sqcup [r,s]) = (\mathbf{a} \sqcup [r-1,s-1], (r-1)^\vee)
  \]
  and
  \[
    R(\mathbf{a} \sqcup [r-1,s-1], (r-1)^\vee) = (s^\vee, \mathbf{a} \sqcup [r,s]).
  \]
\end{lem}

\begin{proof}
  With the same notation as Definition \ref{def: comb R} \eqref{item: def comb R 2}, we see inductively that
  \[
    i_1 = a_1, \dots, i_k = a_k, i_{k+1} = r, i_{k+2} = r+1, \dots, i_{k+s-r} = s-1, \text{ and } i_{k+s-r+1} = r-1.
  \]
  Hence, the first assertion follows.
  Now, the second assertion follows from Proposition \ref{prop: Rcomb is bij}.
\end{proof}

\subsection{Some properties of removable entries}
In this subsection, we fix $l \in [0,2n]$ and $\mathbf{a} = (a_1,\dots,a_l) \in \SST_{2n}(\varpi_l)$.

\begin{lem}\label{lem: rem(a) = rem(a1,...,ai-1)}
  Let $i \in [1,l]$ be such that $a_j \notin \Rem(\mathbf{a})$ for all $j \in [i,l]$.
  Then, we have
  \[
    \Rem(\mathbf{a}) = \Rem(a_1,\dots,a_{i-1}).
  \]
\end{lem}

\begin{proof}
  The assertion is deduced by iterative applications of Proposition \ref{prop: properties of rem} \eqref{item: properties of rem 2}.
\end{proof}

\begin{prop}\label{prop: characterization of rem(a)}
  For each $i \in [1,l]$, we have $a_i \in \Rem(\mathbf{a})$ if and only if one of the following hold:
  \begin{enumerate}
    \item $a_i \notin 2\mathbb{Z}$, $i < l$, $a_{i+1} = a_i+1$, and $a_i < 2i - |\Rem(a_1,\dots,a_{i-1})|$.
    \item $a_i \in 2\mathbb{Z}$, $i > 1$, $a_{i-1} = a_i-1$, and $a_i < 2i - |\Rem(a_1,\dots,a_{i-2})|-1$.
  \end{enumerate}
\end{prop}

\begin{proof}
  We prove the assertion by induction on $l$.
  If $l \leq 1$, then the assertion is clear since we have $\Rem(\mathbf{a}) = \emptyset$ and neither $i < l$ nor $i > 1$ for all $i \in [1,l]$.
  Hence, assume that $l > 1$ and the assertion holds for $0,1,\dots,l-1$.

  First, suppose that
  \begin{align}\label{eq: 2}
    a_l \in 2\mathbb{Z},\ a_{l-1} = a_l-1, \text{ and } a_l < 2l - |\Rem(a_1,\dots,a_{l-2})| - 1.
  \end{align}
  By Definition \ref{def: rem}, we have
  \begin{align}\label{eq: 1}
    \Rem(\mathbf{a}) = \Rem(a_1,\dots,a_{l-2}) \sqcup \{ a_{l-1}, a_l \}.
  \end{align}

  When $i \leq l-2$, equation \eqref{eq: 1} implies that $a_i \in \Rem(\mathbf{a})$ if and only if $a_i \in \Rem(a_1,\dots,a_{l-2})$.
  Hence, the assertion follows from our induction hypothesis.

  When $i \geq l-1$, equation \eqref{eq: 1} implies that $a_i \in \Rem(\mathbf{a})$.
  On the other hand, condition \eqref{eq: 2} implies the first (resp., second) condition in the statement when $i = l-1$ (resp., $i = l$).
  Therefore, the assertion follows in this case.

  Next, suppose that condition \eqref{eq: 2} fails.
  By Definition \ref{def: rem}, it holds that
  \[
    \Rem(\mathbf{a}) = \Rem(a_1,\dots,a_{l-1}).
  \]
  Then, the assertion follows from our induction hypothesis.
\end{proof}

For each $a \in \mathbb{Z}$, set
\begin{align}\label{eq: switch}
  s(a) := \begin{cases}
    a+1 & \text{ if } a \notin 2\mathbb{Z}, \\
    a-1 & \text{ if } a \in 2\mathbb{Z}
  \end{cases}
\end{align}

\begin{lem}\label{lem: rem is invariant under switch}
  Let $a \in [1,2n]$.
  Then, we have $a \in \Rem(\mathbf{a})$ if and only if $s(a) \in \Rem(\mathbf{a})$.
  Consequently, we have $|\Rem(\mathbf{a})| \in 2\mathbb{Z}$.
\end{lem}

\begin{proof}
  The assertion follows from Proposition \ref{prop: characterization of rem(a)}.
\end{proof}

\begin{lem}\label{lem: tail of rem(a)}
  Suppose that $a_l \in 2\mathbb{Z}$ and that there exists $r \in [0,l-1]$ such that $a_{l-r} = a_l -r \in \Rem(\mathbf{a})$.
  Then, we have $a_{l-t} \in \Rem(\mathbf{a})$ for all $t \in [0,r]$.
  Moreover, if $r \notin 2\mathbb{Z}$, then it holds that
  \[
    \Rem(\mathbf{a}) = \Rem(a_1,\dots,a_{l-r-1}) \sqcup [a_l-r, a_l].
  \]
\end{lem}

\begin{proof}
  Since $a_{l-r} = a_l-r$ and $a_{l-r} < a_{l-r+1} < \dots < a_l$, we have $a_{l-t} = a_l-t$ for all $t \in [0,r]$.

  Let us prove for each $t \in [1,r]$ that if $a_{l-t} \in \Rem(\mathbf{a})$, then $a_{l-t+1} \in \Rem(\mathbf{a})$.
  It is clear that this claim implies the first assertion.

  First, suppose that $t \notin 2\mathbb{Z}$.
  Then we have
  \[
    a_{l-t+1} = a_{l-t}+1 = s(a_{l-t}) \in \Rem(\mathbf{a})
  \]
  by Lemma \ref{lem: rem is invariant under switch}.

  Next, suppose that $t \in 2\mathbb{Z}$.
  By Proposition \ref{prop: characterization of rem(a)}, we have $a_{l-t-1} = a_{l-t}-1$ and
  \[
    a_{l-t} < 2(l-t) - |\Rem(a_1,\dots,a_{l-t-2})| -1.
  \]
  Then, by Definition \ref{def: rem}, we obtain
  \begin{align}\label{eq: 3}
    |\Rem(a_1,\dots,a_{l-t})| = |\Rem(a_1,\dots,a_{l-t-2})| + 2.  
  \end{align}

  On the other hand, since $t \in [1,r] \cap 2\mathbb{Z}$, we see that $l-t+1 < l$.
  Hence, we have $a_{l-t+2} = a_{l-t+1} + 1$ and
  \[
    a_{l-t+1} = a_{l-t}+1 < 2(l-t) - |\Rem(a_1,\dots,a_{l-t-2})| = 2(l-t+1) - |\Rem(a_1,\dots,a_{l-t})|,
  \]
  by equation \eqref{eq: 3}.
  Then, Proposition \ref{prop: characterization of rem(a)} implies
  \[
    a_{l-t+1} \in \Rem(\mathbf{a}),
  \]
  as desired.

  So far, we have proved the first assertion.
  The second assertion follows from the following equality, which is obtained from Definition \ref{def: rem} under our hypothesis that $r \notin 2\mathbb{Z}$:
  \[
    \Rem(a_1,\dots,a_{l-r+2k+1}) = \Rem(a_1,\dots,a_{l-r+2k-1}) \sqcup \{ a_{l-r+2k}, a_{l-r+2k+1} \}
  \]
  for all $k \in [0,(r-1)/2]$.
  Thus, we complete the proof.
\end{proof}

\begin{lem}\label{lem: rem(a r,s)}
  Suppose that $a_l \in 2\mathbb{Z}$ and that there exists $r \in [1,l-1]$ such that $a_{l-r} = a_l-r \notin 2\mathbb{Z}$.
  Let $t$ denote the maximal odd integer in $[1,r]$ satisfying
  \[
    a_l-t < 2(l-t) - |\Rem(a_1,\dots,a_{l-r-1})|;
  \]
  if such $t$ does not exist, we set $t := -1$.
  Then, we have
  \[
    \Rem(\mathbf{a}) = \Rem(a_1,\dots,a_{l-r-1}) \sqcup [a_l-t,a_l].
  \]
\end{lem}

\begin{proof}
  Let $t' \in [0,r]$ denote the maximal odd integer such that $a_{l-t'} \in \Rem(\mathbf{a})$; when such $t'$ does not exist, we set $t' := -1$.
  Then, by Lemma \ref{lem: tail of rem(a)}, we have
  \[
    \Rem(\mathbf{a}) = \Rem(a_1,\dots,a_{l-t'-1}) \sqcup [a_l-t', a_l].
  \]
  Moreover, since $a_{l-r}, a_{l-r+1}, \dots, a_{l-t'-1} \notin \Rem(\mathbf{a})$ by the definition of $t'$, Lemma \ref{lem: rem(a) = rem(a1,...,ai-1)} implies that
  \begin{align}\label{eq: 4}
    \Rem(a_1,\dots,a_{l-t'-1}) = \Rem(a_1,\dots,a_{l-r-1}).
  \end{align}
  Hence, in order to complete the proof, we only need to show that $t' = t$.

  First, suppose that $t' \neq -1$.
  Then, by Proposition \ref{prop: characterization of rem(a)} and equation \eqref{eq: 4}, we have
  \[
    a_l-t' = a_{l-t'} < 2(l-t') - |\Rem(a_1,\dots,a_{l-t'-1})| = 2(l-t') - |\Rem(a_1,\dots,a_{l-r-1})|.
  \]
  This implies that $t' \leq t$.
  In particular, $t \neq -1$.
  Then, the definition of $t$ and equation \eqref{eq: 4} imply that
  \[
    a_{l-t} = a_l-t < 2(l-t) - |\Rem(a_1,\dots,a_{l-r-1})| = 2(l-t) - |\Rem(a_1,\dots,a_{l-t-1})|.
  \]
  Hence, we obtain $t \leq t'$.
  Therefore, we conclude that $t' = t$, as desired in this case (when $t' \neq -1$).

  Next, suppose that $t' = -1$.
  We only need to show that there exists no odd integer $t \in [0,r]$ such that
  \[
    a_l-t < 2(l-t) - |\Rem(a_1,\dots,a_{l-r-1})|.
  \]
  Assume contrary.
  Then, by Proposition \ref{prop: characterization of rem(a)}, we see that $a_{l-t} \in \Rem(\mathbf{a})$.
  This contradicts that $t' = -1$.
  Therefore, we obtain $t' = t$, as desired.
\end{proof}

\begin{lem}\label{lem: division of rem(a)}
  Let $i \in [1,l]$ be such that $a_i \notin \Rem(\mathbf{a})$.
  Then, we have
  \[
    \Rem(\mathbf{a}) = \Rem(a_1,\dots,a_{i-1}) \sqcup ((a_{i+1},\dots,a_l) \cap \Rem(\mathbf{a})).
  \]
\end{lem}

\begin{proof}
  By Proposition \ref{prop: properties of rem} \eqref{item: properties of rem 3}, the right-hand side is contained in the left-hand side.
  Hence, we only need to show the opposite containment.

  Let $j \in [1,l]$ be such that $a_j \in \Rem(\mathbf{a})$.
  By our assumption, it must hold that $j \neq i$.

  First, suppose that $j < i-1$.
  In this case, we have $a_j \in \Rem(a_1,\dots,a_{i-1})$ by Proposition \ref{prop: characterization of rem(a)}.

  Next, suppose that $j = i-1$.
  In this case, we have $a_j \in 2\mathbb{Z}$; otherwise, it holds that $a_i = a_{j+1} = a_j+1$ by Proposition \ref{prop: characterization of rem(a)} and $a_j+1 = s(a_j) \in \Rem(\mathbf{a})$ by Lemma \ref{lem: rem is invariant under switch}, which contradicts that $a_i \notin \Rem(\mathbf{a})$.
  Then, we obtain $a_j \in \Rem(a_1,\dots,a_{i-1})$ by Proposition \ref{prop: characterization of rem(a)}.

  Finally, suppose that $j > i$.
  In this case, it is clear that
  \[
    a_j \in (a_{i+1}, \dots, a_l) \cap \Rem(\mathbf{a}).
  \]
  Thus, we complete the proof.
\end{proof}

\begin{prop}\label{prop: low bound for rem}
  We have
  \[
    0 \leq l - |\Rem(\mathbf{a})| \leq \min(l, 2n-l).
  \]
  In particular, it holds that
  \[
    |\Rem(\mathbf{a})| \geq 2(l-n).
  \]
\end{prop}

\begin{proof}
  Set $r := |\Rem(\mathbf{a})|$.
  Since $\Rem(\mathbf{a})$ is a subset of $\{ a_1,\dots, a_l \}$, we have
  \[
    0 \leq l-r \leq l.
  \]
  Hence, we only need to show that $l-r \leq 2n-l$, or equivalently,
  \[
    r \geq 2(l-n).
  \]
  We prove this inequality by induction on $l$.
  When $l \leq 1$, the claim is trivial since we have $\Rem(\mathbf{a}) = \emptyset$ and $n \geq 1$.
  
  Assume that $l > 1$ and our claim is true for $0,1,\dots,l-1$.
  Set
  \[
    \mathbf{a}' := (a_1,\dots, a_{l-1}) \text{ and } r' := |\Rem(\mathbf{a}')|.
  \]
  Then, we have $r \geq r'$. Note that $a_{l-1} \leq 2n-1$.

  Suppose first that $a_{l-1} < 2n-1$.
  Then, we have
  \[
    \mathbf{a}' \in \SST_{2(n-1)}(\varpi_{l-1}).
  \]
  By our induction hypothesis, we obtain
  \[
    r \geq r' \geq 2((l-1) - (n-1)) = 2(l-n),
  \]
  as desired.

  Next, suppose that $a_{l-1} = 2n-1$ and $a_{l-1} \in \Rem(\mathbf{a})$.
  In this case, Lemma \ref{lem: rem is invariant under switch} and Definition \ref{def: rem} imply that $a_l = 2n \in \Rem(\mathbf{a})$ and
  \[
    \Rem(\mathbf{a}) = \Rem(a_1,\dots,a_{l-2}) \sqcup \{ a_{l-1}, a_l \},
  \]
  respectively.
  Set
  \[
    r'' := |\Rem(a_1,\dots,a_{l-2})|.
  \]
  Then, by the same argument as in the previous case, we obtain
  \[
    r = r''+2 \geq 2((l-2)-(n-1)) + 2 = 2(l-n),
  \]
  as desired.

  Finally, suppose that $a_{l-1} = 2n-1$ and $a_{l-1} \notin \Rem(\mathbf{a})$.
  In this case, Proposition \ref{prop: characterization of rem(a)} implies
  \[
    2n-1 = a_{l-1} \geq 2(l-1) - r''.
  \]
  Furthermore, since $r'' \in 2\mathbb{Z}$ by Lemma \ref{lem: rem is invariant under switch}, it follows that
  \[
    2n-2 \geq 2(l-1) - r'',
  \]
  equivalently,
  \[
    r'' \geq 2(l-n).
  \]
  Now, the assertion follows from the fact that $r \geq r''$.
\end{proof}

\begin{cor}\label{cor: low bound for rem(a')}
  Let $i \in [1,l]$.
  \begin{enumerate}
    \item\label{item: low bound for rem(a') 1} If $a_i$ is odd, then $|\Rem(a_1,\dots,a_{i-1})| \geq 2i-a_i-1$.
    \item\label{item: low bound for rem(a') 2} If $a_i$ is even, then $|\Rem(a_1,\dots,a_i)| \geq 2i-a_i$.
  \end{enumerate}
\end{cor}

\begin{proof}
  Observe that $(a_1,\dots,a_{i-1}) \in \SST_{a_i-1}(\varpi_{i-1})$ in the first case, while $(a_1,\dots,a_i) \in \SST_{a_i}(\varpi_i)$ in the second case.
  Then, the assertion follows from Proposition \ref{prop: low bound for rem}.
\end{proof}

\subsection{Some properties of the reduction map}
In this subsection, we fix $l \in [0,2n]$ and $\mathbf{a} = (a_1,\dots,a_l) \in \SST_{2n}(\varpi_l)$.

\begin{lem}\label{lem: red 1,l}
  Let $l \in [1,2n]$ be an even integer.
  Then, we have
  \[
    \red(1,2,\dots,l) = ().
  \]
\end{lem}

\begin{proof}
  By Corollary \ref{cor: low bound for rem(a')} \eqref{item: low bound for rem(a') 2}, we have
  \[
    |\Rem(1,2,\dots,l)| \geq l.
  \]
  This implies that
  \[
    \Rem(1,2,\dots,l) = [1,l].
  \]
  Therefore, the assertion follows.
\end{proof}

\begin{lem}\label{lem: red a with al not in rem a}
  Let $i \in [1,l]$ be such that $a_j \notin \Rem(\mathbf{a})$ for all $j \in [i,l]$.
  Then, we have
  \[
    \red(\mathbf{a}) = \red(a_1,\dots,a_{i-1}) \sqcup (a_i,\dots,a_l).
  \]
\end{lem}

\begin{proof}
  The assertion follows from Lemma \ref{lem: rem(a) = rem(a1,...,ai-1)}.
\end{proof}

\begin{lem}\label{lem: tail of red(a)}
  Suppose that $a_l \in 2\mathbb{Z}$ and that there exists $r \in [0,l-1]$ such that $r \notin 2\mathbb{Z}$ and $a_{l-r} = a_l -r \in \Rem(\mathbf{a})$.
  Then, we have
  \[
    \red(\mathbf{a}) = \red(a_1,\dots,a_{l-r-1}).
  \]
\end{lem}

\begin{proof}
  The assertion follows from Lemma \ref{lem: tail of rem(a)}.
\end{proof}

\begin{lem}\label{lem: red(a r,s)}
  Suppose that $a_l \in 2\mathbb{Z}$ and that there exists $r \in [1,l-1]$ such that $a_{l-r} = a_l-r \notin 2\mathbb{Z}$.
  Then, we have
  \[
    \red(\mathbf{a}) = \red(a_1,\dots,a_{l-r-1}) \sqcup [a_l-r, a_l-t-1],
  \]
  where $t \in [0,r]$ is the maximal odd integer such that
  \[
    a_l-t < 2(l-t) - |\Rem(a_1,\dots,a_{l-r-1})|;
  \]
  when such $t$ does not exist, we set $t := -1$.
\end{lem}

\begin{proof}
  The assertion follows from Lemma \ref{lem: rem(a r,s)}.
\end{proof}

\begin{lem}\label{lem: i_t}
  Let us write $\red(\mathbf{a}) = (a_{i_1}, a_{i_2}, \dots, a_{i_k})$ for some $k \in [0,l]$ and $1 \leq i_1 < i_2 < \cdots < i_k \leq l$.
  Then, for each $t \in [1,k]$, we have
  \[
    i_t = t + |\Rem(a_1,a_2,\dots,a_{i_t-1})|.
  \]
\end{lem}

\begin{proof}
  Clearly, we have
  \[
    i_t = t + |(a_1,\dots,a_{i_t}) \cap \Rem(\mathbf{a})|.
  \]
  Since $a_{i_t} \notin \Rem(\mathbf{a})$, Lemma \ref{lem: division of rem(a)} implies
  \[
    (a_1,\dots,a_{i_t}) \cap \Rem(\mathbf{a}) = \Rem(a_1,\dots,a_{i_t-1}).
  \]
  Hence, the assertion follows.
\end{proof}

\begin{prop}\label{prop: red(a) is symplectic}
  The tableau $\red(\mathbf{a})$ is symplectic.
  Consequently, the assignment $\mathbf{a} \mapsto \red(\mathbf{a})$ gives rise to a map
  \[
    \red: \SST_{2n}(\varpi_l) \rightarrow \bigsqcup_{\substack{0 \leq k \leq \min(l, 2n-l) \\ l-k \in 2\mathbb{Z}}} \SpT_{2n}(\varpi_k)
  \]
  for each $l \in [0,2n]$.
\end{prop}

\begin{proof}
  Let us write $\red(\mathbf{a}) = (a_{i_1}, \dots, a_{i_k})$ for some $k \leq l$ and $1 \leq i_1 < \cdots < i_k \leq l$.
  Assume contrary that $\red(\mathbf{a})$ is not symplectic.
  Then, by Lemma \ref{lem: existence of removable entries}, there exists $r \in [2,k]$ such that
  \[
    a_{i_r} = 2r-2,\ a_{i_{r-1}} = 2r-3,
  \]
  and
  \[
    a_{i_s} \geq 2s-1 \ \text{ for all } s \in [1,r-2].
  \]
  In particular, we have
  \[
    i_{r-1} = i_r-1.
  \]
  
  By Lemma \ref{lem: i_t}, it holds that
  \[
    i_{r-1} = (r-1) + |\Rem(a_1,\dots,a_{i_r-2})|.
  \]
  Therefore, we have
  \[
    2i_{r-1} - |\Rem(a_1,\dots,a_{i_r-2})| = 2(r-1) + |\Rem(a_1,\dots,a_{i_r-2})| \geq 2r-2.
  \]

  So far, we have shown that $a_{i_r}$ is even, $a_{i_r-1} = a_{i_r}-1$, and
  \[
    a_{i_r} = 2r-2 < 2i_r - |\Rem(a_1,\dots,a_{i_r-2})|-1.
  \]
  Then, Proposition \ref{prop: characterization of rem(a)} implies that $a_{i_r} \in \Rem(\mathbf{a})$.
  However, this contradicts that $a_{i_r}$ is an entry of $\red(\mathbf{a})$.
  Thus, we complete the proof.
\end{proof}

\begin{prop}\label{prop: red(a) = a iff a is symplectic}
  We have $\red(\mathbf{a}) = \mathbf{a}$ if and only if $\mathbf{a}$ is symplectic.
\end{prop}

\begin{proof}
  First, suppose that $\red(\mathbf{a}) = \mathbf{a}$.
  Then, Proposition \ref{prop: red(a) is symplectic} implies that the tableaux $\mathbf{a}$ is symplectic.

  Next, suppose that $\mathbf{a}$ is symplectic.
  Assume contrary that $\red(\mathbf{a}) \neq \mathbf{a}$.
  Then, by Proposition \ref{prop: characterization of rem(a)}, there exists $i \in [2,l]$ such that $a_i < 2i-1$.
  This contradicts that $\mathbf{a}$ is symplectic.
  Thus, the assertion follows.
\end{proof}

\begin{cor}\label{cor: suc(T) = T}
  Let $\nu \in \Par_{\leq n}$ and $T \in \SST_{2n}(\nu)$.
  Then, we have $\suc(T) = T$ if and only if $T$ is symplectic.
\end{cor}

\begin{proof}
  Let us write $d(T) = (\mathbf{a}, T')$.
  By Lemma \ref{lem: stationality of suc} and its proof, we have $\suc(T) = T$ if and only if $\red(\mathbf{a}) = \mathbf{a}$.
  The latter is equivalent to that $\mathbf{a}$ is symplectic, which is then equivalent to that $T$ is symplectic.
  Thus, the assertion follows.
\end{proof}

\subsection{Factorization of the reduction map}
Recall the combinatorial $R$-matrices, the map $s$, and the map $\cdot^\vee$ from Definition \ref{def: comb R}, \eqref{eq: switch}, and \eqref{eq: a^vee}, respectively.

\begin{prop}\label{prop: pre-factorization of red}
  Let $l \in [2,2n]$, $\mathbf{a} = (a_1,\dots,a_l) \in \SST_{2n}(\varpi_l)$.
  Let us write
  \[
    R(s(a_l)^\vee, (a_1,\dots,a_{l-1})) = (\mathbf{c}, \mathbf{d})
  \]
  for some $(\mathbf{c}, \mathbf{d}) \in \SST_{2n}(\varpi_{l-1}) \times \SST_{2n}(\varpi_{2n-1})$, and
  \[
    R(\red(\mathbf{c}), \mathbf{d}) = (s(b_k)^\vee, \mathbf{b}')
  \]
  for some $b_k \in \mathbb{Z}$ and $\mathbf{b}' = (b_1,\dots,b_{k-1}) \in \SST_{2n}(\varpi_{k-1})$, where $k := |\red(\mathbf{c})| + 1$.
  Then, we have $b_k > b_{k-1}$ and
  \[
    \mathbf{b} := (b_1,\dots,b_k) = \begin{cases}
      (1,2) & \text{ if } a_l = l \in 2\mathbb{Z}, \\
      \red(\mathbf{a}) & \text{ otherwise}.
    \end{cases}
  \]
\end{prop}

\begin{proof}
  Set $\mathbf{a}' := (a_1,\dots,a_{l-1})$.

  First, suppose that $a_l = l \in 2\mathbb{Z}$.
  Then, we have
  \begin{align*}
    \begin{split}
      (s(a_l)^\vee, (1,\dots,l-1)) &= ((l-1)^\vee, \mathbf{a}') \\
      &\overset{R}{\mapsto} ((1,\dots,l-2,2n), (2n)^\vee) = (\mathbf{c}, \mathbf{d}) \\
      &\overset{(\red, \mathrm{id})}{\mapsto} ((2n), (2n)^\vee) = (\red(\mathbf{c}), \mathbf{d}) \\
      &\overset{R}{\mapsto} (1^\vee, (1)) = (s(b_k)^\vee, \mathbf{b}').
    \end{split}
  \end{align*}
  Here, we used Lemmas \ref{lem: Rcomb tvee 1,t}, \ref{lem: red a with al not in rem a}, \ref{lem: red 1,l}, and \ref{lem: Rcomb m mvee}.
  As a result, we obtain $\mathbf{b} = (1,2)$, as desired.

  Next, suppose that $a_l \notin 2\mathbb{Z}$.
  Then, we have
  \begin{align*}
    \begin{split}
      (s(a_l)^\vee, \mathbf{a}') &= ((a_l+1)^\vee, \mathbf{a}') \\
      &\overset{R}{\mapsto} (\mathbf{a}', (a_l+1)^\vee) = (\mathbf{c}, \mathbf{d}) \\
      &\overset{(\red, \mathrm{id})}{\mapsto} (\red(\mathbf{a}'), (a_l+1)^\vee) = (\red(\mathbf{c}), \mathbf{d}) \\
      &\overset{R}{\mapsto} ((a_l+1)^\vee, \red(\mathbf{a}')) = (s(b_k)^\vee, \mathbf{b}').
    \end{split}
  \end{align*}
  Here, we used Lemma \ref{lem: Rcomb a sub b}.
  This implies that $\mathbf{b}' = \red(\mathbf{a}')$ and $b_k = s(a_l+1) = a_l$.
  Hence, we obtain
  \[
    b_k = a_l > a_{l-1} \geq b_{k-1}
  \]
  since $\mathbf{b}' = \red(\mathbf{a}')$ is a subsequence of $\mathbf{a}' = (a_1,\dots,a_{l-1})$.

  On the other hand, since $a_l+1 \notin \Rem(\mathbf{a})$, we have $a_l = s(a_l+1) \notin \Rem(\mathbf{a})$ by Lemma \ref{lem: rem is invariant under switch}.
  Then, Lemma \ref{lem: red a with al not in rem a} implies
  \[
    \red(\mathbf{a}) = \red(\mathbf{a}') \sqcup \{ a_l \} = \mathbf{b},
  \]
  as desired.

  Finally, suppose that $a_l \in 2\mathbb{Z}$ and $a_l \neq l$.
  Set
  \[
    r := \max\{ i \in [0,l-1] \mid a_{l-i} = a_l-i \}, \quad \mathbf{a}'' := (a_1,\dots,a_{l-r-1}).
  \]
  Note that $r < l-1$.
  Then, we have $\mathbf{a} = \mathbf{a}'' \sqcup [a_l-r,a_l]$, $\mathbf{a}' = \mathbf{a}'' \sqcup [a_l-r,a_l-1]$, $s(a_l) = a_l-1$, and
  \[
    (\mathbf{c}, \mathbf{d}) = R((a_l-1)^\vee, \mathbf{a}') = ((\mathbf{a}'' \sqcup [a_l-r-1,a_l-2]), (a_l-r-1)^\vee).
  \]
  The last equality follows from Lemma \ref{lem: Rcomb svee a r,s}.

  \begin{enumerate}
    \item When $r \in 2\mathbb{Z}$.
    Since $a_l-r-1 \notin 2\mathbb{Z}$ and $a_l-2 \in 2\mathbb{Z}$, we can apply Lemma \ref{lem: red(a r,s)} to obtain
    \[
      \red(\mathbf{c}) = \red(\mathbf{a}'') \sqcup [a_l-r-1, a_l-t-3],
    \]
    where $t$ is the maximal odd integer in $[0, r-1]$ such that
    \[
      (a_l-2)-t < 2((l-1)-t) - |\Rem(\mathbf{a}'')|;
    \]
    when such $t$ does not exist, we set $t := -1$.
    Then, by Lemma \ref{lem: Rcomb svee a r,s} again, we have
    \begin{align*}
      \begin{split}
        R(\red(\mathbf{c}), \mathbf{d}) = ((a_l-t-2)^\vee, \red(\mathbf{a}'') \sqcup [a_l-r, a_l-t-2]).
      \end{split}
    \end{align*}
    This implies that $\mathbf{b}' = \red(\mathbf{a}'') \sqcup [a_l-r, a_l-t-2]$, $b_k = s(a_l-t-2) = a_l-t-1$, and
    \[
      \mathbf{b} = \red(\mathbf{a}'') \sqcup [a_l-r, a_l-t-1].
    \]

    On the other hand, we have
    \[
      \mathbf{a} = \mathbf{a}'' \sqcup [a_l-r, a_l] = (\mathbf{a}'' \sqcup (a_l-r)) \sqcup [a_l-r+1, a_l].
    \]
    By Lemma \ref{lem: red(a r,s)}, we have
    \[
      \red(\mathbf{a}) = \red(\mathbf{a}'' \sqcup (a_l-r)) \sqcup [a_l-r+1, a_l-t'-1],
    \]
    where $t'$ is the maximal odd integer in $[0,r-1]$ such that
    \[
      a_l-t' < 2(l-t') - |\Rem(\mathbf{a}'' \sqcup (a_l-r))|;
    \]
    when such $t'$ does not exist, we set $t' := -1$.
    Since $s(a_l-r) = a_l-r-1 \notin \mathbf{a}''$, we have
    \[
      \Rem(\mathbf{a}'' \sqcup (a_l-r)) = \Rem(\mathbf{a}'') \text{ and } \red(\mathbf{a}'' \sqcup (a_l-r)) = \red(\mathbf{a}'') \sqcup (a_l-r).
    \]
    Therefore, we obtain
    \[
      \red(\mathbf{a}) = \red(\mathbf{a}'') \sqcup [a_l-r, a_l-t'-1].
    \]

    Now, one can straightforwardly verify that $t' = t$.
    Hence, we conclude that
    \[
      \red(\mathbf{a}) = \red(\mathbf{a}'') \sqcup [a_l-r, a_l-t-1] = \mathbf{b},
    \]
    as desired.

    \item When $r \notin 2\mathbb{Z}$  and $a_l-r-1 \notin \Rem(\mathbf{c})$.
    This case can be proved as in the same way as the previous case.

    \item When $r \notin 2\mathbb{Z}$ and $a_l-r-1 \in \Rem(\mathbf{c})$.
    Since $a_l-r-2 = s(a_l-r-1) \in \Rem(\mathbf{c})$, we see that
    \[
      a_{l-r-1} = a_l-r-2.
    \]
    Set
    \[
      \mathbf{a}''' := (a_1,\dots,a_{l-r-2}).
    \]
    Then, we have
    \[
      \mathbf{c} = \mathbf{a}''' \sqcup [a_l-r-2, a_l-2]
    \]
    with $a_l-r-2 \in \Rem(\mathbf{c})$.
    By Lemma \ref{lem: tail of red(a)}, we obtain
    \[
      \red(\mathbf{c}) = \red(\mathbf{a}''').
    \]
    Then, Lemma \ref{lem: Rcomb a sub b} implies
    \[
      R(\red(\mathbf{c}), \mathbf{d}) = (\mathbf{d}, \red(\mathbf{c})),
    \]
    and consequently,
    \[
      \mathbf{b} = \red(\mathbf{a}''') \sqcup (a_l-r-2).
    \]

    On the other hand, we have
    \[
      \mathbf{a} = (\mathbf{a}''' \sqcup (a_l-r-2)) \sqcup [a_l-r, a_l].
    \]
    Since $a_l-r-2 \notin 2\mathbb{Z}$, Proposition \ref{prop: properties of rem} \eqref{item: properties of rem 2} implies that
    \[
      \Rem(\mathbf{a}''' \sqcup (a_l-r-2)) = \Rem(\mathbf{a}''').
    \]
    On the other hand, the fact that $a_l-r-2 \in \Rem(\mathbf{c})$, together with Proposition \ref{prop: characterization of rem(a)}, implies
    \[
      a_l-r-2 < 2(l-r-1) - |\Rem(\mathbf{a}''')|.
    \]
    These two equalities and Proposition \ref{prop: characterization of rem(a)} ensure that
    \[
      a_l-r \in \Rem(\mathbf{a}).
    \]
    Then, Lemma \ref{lem: tail of red(a)} implies
    \[
      \red(\mathbf{a}) = \red(\mathbf{a}''') \sqcup (a_l-r-2) = \mathbf{b},
    \]
    as desired.
  \end{enumerate}
\end{proof}

\begin{ex}
  Let $n = 7$ and $\mathbf{a} = (1,3,4,5,6,7,11,12,13,14)$.
  Then, $s(14)^\vee = (1,\dots,12,14)$, and
  \[
    R(s(14)^\vee,(1,3,4,5,6,7,11,12,13)) = ((1,3,4,5,6,7,10,11,12),(1,\dots,9,11,\dots,14)).
  \]
  Since
  \[
    \red(1,3,4,5,6,7,10,11,12) = (1,7,10),
  \]
  we compute
  \[
    R((1,7,10),(1,\dots,9,11,\dots,14)) = ((1,\dots,10,12,13,14),(1,7,11)).    
  \]
  Therefore, we obtain
  \[
    \red(\mathbf{a}) = (1,7,11,12).
  \]
\end{ex}

In order to state the main result in this section, let us introduce four maps:
\begin{itemize}
  \item $\bigvee: \SST_{2n}(\varpi_l) \rightarrow \SST_{2n}(\varpi_1) \times \SST_{2n}(\varpi_{l-1});\ (a_1,\dots,a_l) \mapsto ((a_l), (a_1,\dots,a_{l-1}))$.
  \item $K: \SST_{2n}(\varpi_1) \rightarrow \SST_{2n}(\varpi_{2n-1});\ (a) \mapsto s(a)^\vee$.
  \item \begin{align*}
    \begin{split}
      \bigwedge: \SST_{2n}(\varpi_1) \times \SST_{2n}(\varpi_k) &\rightarrow \SST_{2n}(\varpi_{k+1}) \sqcup \{ 0 \};\\ ((a), (a_1,\dots,a_k)) &\mapsto \begin{cases}
        (a_1,\dots,a_k,a) & \text{ if } a > a_k, \\
        0 & \text{ if } a \leq a_k,
      \end{cases}
    \end{split}
  \end{align*}
  where $0$ is a formal symbol.
  \item $\pi: \SST_{2n}(\varpi_{k'}) \rightarrow \bigsqcup_{k''=0}^{2n} \SST_{2n}(\varpi_{k''});\ \mathbf{a} \mapsto \begin{cases}
    () & \text{ if } \mathbf{a} = (1,2), \\
    \mathbf{a} & \text{ if } \mathbf{a} \neq (1,2).
  \end{cases}$
\end{itemize}

\begin{cor}\label{cor: factorization of red}
  Let $l \in [0,2n]$.
  \begin{enumerate}
    \item\label{item: factorization of red 1} If $l \geq 2$, then the composite
    \[
      \pi \circ \bigwedge \circ (K^{-1}, \mathrm{id}) \circ R \circ (\red, \mathrm{id}) \circ R \circ (K, \mathrm{id}) \circ \bigvee
    \]
    is well defined on $\SST_{2n}(\varpi_l)$ and coincides with the reduction map.
    \item\label{item: factorization of red 2} The reduction map on $\SST_{2n}(\varpi_l)$ is injective.
  \end{enumerate}
\end{cor}

\begin{proof}
  The first assertion follows from Proposition \ref{prop: pre-factorization of red} and Lemma \ref{lem: red 1,l}.
  The second assertion for $l \leq 1$ is trivial, and for $l \geq 2$ follows from the first one since each factor in the composite is injective (on a suitable domain).
\end{proof}

\begin{cor}\label{cor: inj of comb suc}
  Let $\lambda \in \Par_{\leq 2n}$.
  Then, the successor map in injective on $\SST_{2n}(\lambda)$.
\end{cor}

\begin{proof}
  The assertion follows from the injectivity of the map $d$ (see \eqref{eq: map d} for the definition), the reduction map (Corollary \ref{cor: factorization of red} \eqref{item: factorization of red 2}), and the Pieri's formula (Proposition \ref{prop: Pieri formula}).
\end{proof}

\section{Preliminaries from Lie algebras}\label{sect: preliminary Lie alg}
In this section, we briefly review representation theory of $\mathfrak{gl}_{2n}(\mathbb{C})$ and $\mathfrak{sp}_{2n}(\mathbb{C})$.

\subsection{General linear algebras}
Let $\mathfrak{g}$ denote the general linear algebra $\mathfrak{gl}_{2n}(\mathbb{C})$, that is, the complex Lie algebra $\mathbb{C}$ consisting of all $2n \times 2n$ complex matrices.

For each $i,j \in [1,2n]$, let $E_{i,j}$ denote the matrix unit with entry $1$ at $(i,j)$ position.
For each $i \in [1,2n]$, set
\[
  d_i := E_{i,i}.
\]
Also, for each $i \in [1,2n-1]$, set
\[
  h_i := d_i - d_{i+1}.
\]

Let $M$ be a finite-dimensional $\mathfrak{g}$-module.
Then, it decomposes into its weight spaces:
\[
  M = \bigoplus_{\mathbf{w} = (w_1,\dots,w_{2n}) \in \mathbb{Z}^{2n}} M_{\mathbf{w}},
\]
  where
\[
  M_\mathbf{w} := \{ m \in M \mid d_i m = w_i m \ \text{ for all } i \in [1,2n] \}.
\]
The character of $M$ is the Laurent polynomial given by
\[
  \operatorname{ch}_\mathfrak{g} M = \sum_{\mathbf{w} \in \mathbb{Z}^{2n}} (\dim M_\mathbf{w}) x_1^{w_1} \cdots x_{2n}^{w_{2n}} \in \mathbb{Z}[x_1^{\pm 1}, \dots, x_{2n}^{\pm 1}].
\]

For each $\lambda \in \Par_{\leq 2n}$, there exists a unique, up to isomorphism, finite-dimensional irreducible $\mathfrak{g}$-module $V^\mathfrak{g}(\lambda)$ such that
\[
  \operatorname{ch}_\mathfrak{g} V^\mathfrak{g}(\lambda) = s_\lambda(x_1,\dots,x_{2n}),
\]
where the right-hand side denotes the Schur function (see \eqref{eq: Schur} for the definition).

\subsection{Symplectic Lie algebras}
Consider the $2n$-dimensional complex vector space $\mathbb{C}^{2n}$ with a standard basis $\{ e_1,\dots,e_{2n }\}$.
Let $\langle ,  \rangle$ denote the skew-symmetric bilinear form on $\mathbb{C}^{2n}$ given by
\[
  \langle e_i, e_j \rangle = \begin{cases}
    1 & \text{ if } j-i = n, \\
    -1 & \text{ if } i-j = n, \\
    0 & \text{ otherwise}.
  \end{cases}
\]

Let $\mathfrak{s}$ denote the symplectic Lie algebra $\mathfrak{sp}_{2n}(\mathbb{C})$, that is, the subalgebra of $\mathfrak{g}$ consisting of $X \in \mathfrak{g}$ such that
\[
  {}^t X J_n + J_n X = O_{2n},
\]
where $O_{2n}$ denote the zero matrix of size $2n$,
\[
  J_n := \begin{pmatrix}
    O_n & I_n \\
    -I_n & O_n
  \end{pmatrix}
\]
the matrix representation of the skew-symmetric bilinear form $\langle ,  \rangle$ with respect to the standard basis with $I_n$ the identity matrix of size $n$.

For each $i \in [1,n]$, set
\[
  h'_i := \begin{cases}
    h_i - h_{n+i} & \text{ if } i < n, \\
    d_n - d_{2n} & \text{ if } i = n.
  \end{cases}
\]
Then, we have $h'_i \in \mathfrak{s}$.

Let $M$ be a finite-dimensional $\mathfrak{s}$-module.
Then, it decomposes into its weight spaces:
\[
  M = \bigoplus_{\mathbf{z} = (z_1,\dots,z_{n}) \in \mathbb{Z}^{n}} M_{\mathbf{z}},
\]
where
\[
  M_\mathbf{z} := \{ m \in M \mid h'_i m = (z_i-z_{i+1})m \ \text{ for all } i \in [1,n-1] \text{ and } h'_n m = z_n m \}.
\]
The character of $M$ is the Laurent polynomial given by
\[
  \operatorname{ch}_\mathfrak{s} M = \sum_{\mathbf{z} \in \mathbb{Z}^{n}} (\dim M_\mathbf{z}) y_1^{z_1} \cdots y_{n}^{z_{n}} \in \mathbb{Z}[y_1^{\pm 1}, \dots, y_n^{\pm 1}].
\]

For each $\nu \in \Par_{\leq n}$, there exists a unique finite-dimensional irreducible $\mathfrak{s}$-module $V^{\mathfrak{s}}(\nu)$ such that
\[
  \operatorname{ch}_\mathfrak{s} V^{\mathfrak{s}}(\nu) = s^{Sp}_\nu(y_1,\dots,y_{n}),
\]
where the right-hand side denotes the symplectic Schur function (see \eqref{eq: sp Schur} for the definition).

Let $M = \bigoplus_{\mathbf{w} \in \mathbb{Z}^{2n}} M_\mathbf{w}$ be a finite-dimensional $\mathfrak{g}$-module.
For each $m \in M_\mathbf{w}$, we have
\[
  h'_i m = \begin{cases}
    (w_i-w_{i+1} - w_{n+i} + w_{n+i+1}) m & \text{ if } i < n, \\
    (w_n - w_{2n})m & \text{ if } i = n,
  \end{cases} \ \text{ for all } i \in [1,n].
\]
Therefore, as an $\mathfrak{s}$-module, the $M$ decomposes as
\[
  M = \bigoplus_{\mathbf{z} \in \mathbb{Z}^n} M_\mathbf{z}, \quad M_{\mathbf{z}} = \bigoplus_{\mathbf{w} \in (\res^\mathfrak{s})^{-1}(\mathbf{z})} M_\mathbf{w},
\]
where
\[
  \res^\mathfrak{s}: \mathbb{Z}^{2n} \rightarrow \mathbb{Z}^n;\ (w_1,\dots,w_{2n}) \mapsto (w_1-w_{n+1}, w_2-w_{n+2}, \dots, w_n - w_{2n}).
\]
This observation implies that
\[
  \operatorname{ch}_{\mathfrak{s}} M = \res^\mathfrak{s}(\operatorname{ch}_\mathfrak{g} M),
\]
where
\[
  \res^\mathfrak{s}: \mathbb{Z}[x_1^{\pm 1}, \dots, x_{2n}^{\pm 1}] \rightarrow \mathbb{Z}[y_1^{\pm 1}, \dots, y_n^{\pm 1}]
\]
denotes the ring homomorphism such that
\[
  \res^\mathfrak{s}(x_i) = \begin{cases}
    y_i & \text{ if } i \leq n, \\
    y_{i-n}^{-1} & \text{ if } i > n.
  \end{cases}
\]
In particular, for each $\lambda \in \Par_{\leq 2n}$, we obtain
\begin{align}\label{eq: ch_s V^g(lm)}
  \operatorname{ch}_\mathfrak{s} V^\mathfrak{g}(\lambda) = s_\lambda(y_1,\dots,y_n,y_1^{-1},\dots,y_n^{-1}).
\end{align}

On the other hand, as an $\mathfrak{s}$-module, $V^\mathfrak{g}(\lambda)$ decomposes into a direct sum of several copies of $V^\mathfrak{s}(\nu)$ for various $\nu \in \Par_{\leq n}$:
\begin{align}\label{eq: decomp of V^g(lm)}
  V^\mathfrak{g}(\lambda) \simeq \bigoplus_{\nu \in \Par_{\leq n}} V^\mathfrak{s}(\nu)^{m_{\lambda,\nu}} \ \text{ for some } m_{\lambda,\nu} \geq 0.
\end{align}
An explicit descriptions of the multiplicities $m_{\lambda,\nu}$ is called a \emph{branching rule}.

\subsection{A Non-standard realization of the symplectic algebra}
Let $\mathfrak{k}$ denote the Lie subalgebra of $\mathfrak{g}$ generated by
\[
  \{ E_{2i-1, 2i},\ E_{2i,2i-1} \mid i \in [1,n] \} \sqcup \{ E_{2i+1,2i} + E_{2i-1,2i+2} \mid i \in [1,n-1] \}.
\]
Then, there exists an isomorphism
\[
  f: \mathfrak{s} \rightarrow \mathfrak{k}
\]
of Lie algebras such that
\[
  f(h'_i) = \begin{cases}
    (-1)^{i-1}(h_{2i-1} + h_{2i+1}) & \text{ if } i < n, \\
    (-1)^{n-1}h_{2n-1} & \text{ if } i = n,
  \end{cases}
\]
({\it cf.}\ \cite[Section 4.3]{W21a}).

\begin{rem}
  A matrix description of $\mathfrak{k}$ can be found in \cite[Remark 5.2]{CuSh22}.
\end{rem}

The notions of weight $\mathfrak{k}$-modules and their characters are defined in the same way as those for $\mathfrak{s}$ by replacing $h'_i$ with $f(h'_i)$.

For each $\nu \in \Par_{\leq n}$, the irreducible $\mathfrak{s}$-module $V^{\mathfrak{s}}(\nu)$ is equipped with a $\mathfrak{k}$-module structure via the isomorphism $f$.
Let $V^\mathfrak{k}(\nu)$ denote this $\mathfrak{k}$-module.
Clearly, we have
\[
  \operatorname{ch}_\mathfrak{k} V^\mathfrak{k}(\nu) = \operatorname{ch}_\mathfrak{s} V^\mathfrak{s}(\nu) = s^{Sp}_\nu(y_1,\dots,y_n).
\]

Let $M$ be a $\mathfrak{g}$-module, $\mathbf{w} = (w_1,\dots,w_{2n}) \in \mathbb{Z}^{2n}$, and $m \in M_\mathbf{w}$.
Then, for each $i \in [1,n]$, we have
\[
  f(h'_i) m = \begin{cases}
    (-1)^{i-1}(w_{2i-1} - w_{2i} + w_{2i+1} - w_{2i+2})m & \text{ if } i < n, \\
    (-1)^{n-1}(w_{2n-1} - w_{2n})m & \text{ if } i = n.
  \end{cases}
\]
Therefore, defining a map
\[
  \res^\mathfrak{k}: \mathbb{Z}^{2n} \rightarrow \mathbb{Z}^n
\]
by
\[
  \res^\mathfrak{k}(w_1,\dots,w_{2n}) = (w_1-w_2, -w_3+w_4, w_5-w_6, \dots, (-1)^{n-1}(w_{2n-1} - w_{2n})),
\]
we obtain the following weight space decomposition of $M$ as a $\mathfrak{k}$-module:
\[
  M = \bigoplus_{\mathbf{z} \in \mathbb{Z}^n} M_\mathbf{z}, \quad M_\mathbf{z} = \bigoplus_{\mathbf{w} \in (\res^\mathfrak{k})^{-1}(\mathbf{z})} M_\mathbf{w}.
\]
Consequently, we have
\begin{align}\label{eq: ch_k M}
  \operatorname{ch}_\mathfrak{k} M = \res^\mathfrak{k}(\operatorname{ch}_\mathfrak{g} M),
\end{align}
where
\[
  \res^\mathfrak{k}: \mathbb{Z}[x_1^{\pm 1}, \dots, x_{2n}^{\pm 1}] \rightarrow \mathbb{Z}[y_1^{\pm 1}, \dots, y_n^{\pm 1}]
\]
denotes the ring homomorphism given by
\[
  \res^\mathfrak{k}(x_{2i-1}) = \begin{cases}
    y_i & \text{ if $i$ is odd}, \\
    y_i^{-1} & \text{ if $i$ is even},
  \end{cases} \quad \res^\mathfrak{k}(x_{2i}) = \begin{cases}
    y_i^{-1} & \text{ if $i$ is odd}, \\
    y_i & \text{ if $i$ is even},
  \end{cases} \quad \text{ for all } i \in [1,n].
\]

\begin{prop}\label{prop: ch_t V^g(lm)}
  We have
  \[
    \operatorname{ch}_\mathfrak{k} V^\mathfrak{g}(\lambda) = \sum_{\nu \in \Par_{\leq n}} m_{\lambda,\nu} s^{Sp}_\nu(y_1,\dots,y_n).
  \]
\end{prop}

\begin{proof}
  Using the fact that the Schur functions are symmetric, we compute as follows:
  \begin{align*}
    \begin{split}
      \operatorname{ch}_\mathfrak{k} V^\mathfrak{g}(\lambda) &\overset{\eqref{eq: ch_k M}}{=} \res^\mathfrak{k}(\operatorname{ch}_\mathfrak{g} V^\mathfrak{g}(\lambda)) \\
      &= s_\lambda(y_1,y_1^{-1},y_2^{-1},y_2,\dots,y_n^{(-1)^{n-1}}, y_n^{(-1)^n}) \\
      &= s_\lambda(y_1,\dots,y_n,y_1^{-1},\dots,y_n^{-1}) \\
      &\overset{\eqref{eq: ch_s V^g(lm)}}{=} \operatorname{ch}_\mathfrak{s} V^\mathfrak{g}(\lambda) \\
      &\overset{\eqref{eq: decomp of V^g(lm)}}{=} \sum_{\nu \in \Par_{\leq n}} m_{\lambda,\nu} s^{Sp}_\nu(y_1,\dots,y_n).
    \end{split}
  \end{align*}
  Thus, the assertion follows.
\end{proof}

\section{Preliminaries from quantum symmetric pairs}\label{sect: preliminary qsp}
In this section, we briefly review representation theory of the quantum group $\mathbf{U}$ of $\mathfrak{gl}_{2n}(\mathbb{C})$ and an $\imath$quantum group $\mathbf{U}^\imath$ of $\mathfrak{sp}_{2n}(\mathbb{C})$.

\subsection{Quantum group of type $A$}
Let $\mathbf{U}$ denote the quantum group $U_q(\mathfrak{gl}_{2n})$, that is, the unital associative algebra over $\mathbb{Q}(q)$ with generators
\[
  E_i, F_i, D_k^{\pm 1} \ \text{ for } i \in [1,2n-1] \text{ and } k \in [1,2n]
\]
subject to the following relations:
\begin{align*}
  \begin{split}
    &D_k D_k^{-1} = D_k^{-1} D_k = 1, \\
    &D_k D_l = D_l D_k, \\
    &D_k E_i = q^{\delta_{k,i} - \delta_{k,i+1}} E_i D_k, \quad D_k F_i = q^{-\delta_{k,i} + \delta_{k,i+1}} F_i D_k, \\
    &E_i F_j - F_j E_i = \delta_{i,j} \frac{K_i-K_i^{-1}}{q-q^{-1}}, \\
    &E_i E_j = E_j E_i, \quad F_i F_j = F_j F_i, \\
    &E_i^2E_j - (q+q^{-1})E_iE_jE_i + E_jE_i^2 = 0, \quad F_i^2F_j - (q+q^{-1})F_iF_jF_i + F_jF_i^2 = 0,
  \end{split}
\end{align*}
where
\[
  K_i := D_iD_{i+1}^{-1}.
\]

A $\mathbf{U}$-module $M$ is said to be a \emph{weight module} if it admits a decomposition
\[
  M = \bigoplus_{\mathbf{w} \in \mathbb{Z}^{2n}} M_\mathbf{w}
\]
as a vector space such that
\[
  M_\mathbf{w} = \{ m \in M \mid D_i m = q^{w_i} m \ \text{ for all } i \in [1,2n] \}.
\]

The character of a finite-dimensional weight $\mathbf{U}$-module $M$ is the Laurent polynomial $\operatorname{ch} M \in \mathbb{Z}[x_1^{\pm 1}, \dots, x_{2n}^{\pm 1}]$ given by
\[
  \operatorname{ch} M = \sum_{\mathbf{w} \in \mathbb{Z}^{2n}} (\dim M_\mathbf{w}) \mathbf{x}^\mathbf{w}.
\]

By \cite[19.1.1]{Lus93}, there exists a unique anti-algebra automorphism $\rho$ on $\mathbf{U}$ such that
\[
  \rho(E_i) = qK_iF_i, \quad \rho(F_i) = qK_i^{-1} E_i, \quad \rho(K_i) = K_i \ \text{ for all } i \in I.
\]
A symmetric bilinear form $(,)$ on a $\mathbf{U}$-module $M$ is said to be contragredient if
\[
  (x m_1, m_2) = (m_1, \rho(x) m_2) \ \text{ for all } x \in \mathbf{U},\ m_1, m_2 \in M.
\]
A basis $B$ of such $M$ is said to be \emph{almost orthonormal} if
\[
  (b_1, b_2) \in \delta_{b_1, b_2} + q^{-1} \mathbb{Q}[\![q^{-1}]\!] \ \text{ for all } b_1,b_2 \in B.
\]
For each $m_1, m_2 \in M$, we write $m_1 \equiv m_2$ to indicate that
\[
  m_1 - m_2 \in q^{-1} \mathbb{Q}[\![q^{-1}]\!] B.
\]

Given two $\mathbf{U}$-modules $M$ and $N$ with almost orthonormal bases $B_M$ and $B_N$ respectively, we see that the tensor product $B_M \otimes B_N := \{ b_1 \otimes b_2 \mid (b_1,b_2) \in B_M \times B_N \}$ forms an almost orthonormal basis of $M \otimes N$ with respect to the natural bilinear form.

For each $\lambda \in \Par_{\leq 2n}$, there exists a unique, up to isomorphism, finite-dimensional irreducible $\mathbf{U}$-module $V(\lambda)$ such that
\[
  \operatorname{ch} V(\lambda) = s_\lambda(x_1,\dots,x_{2n}).
\]
The weight space $V(\lambda)_{\lambda}$ is one-dimensional.
Fix a nonzero vector $v_\lambda \in V(\lambda)_\lambda$.
Then, there exists a unique contragredient form $(,)$ on $V(\lambda)$ such that $(v_\lambda, v_\lambda) = 1$.

There exists a distinguished basis $\mathbf{B}(\lambda)$, called the \emph{canonical basis}, of $V(\lambda)$.
It is of the form
\[
  \mathbf{B}(\lambda) = \{ b_T \mid T \in \SST_{2n}(\lambda) \},
\]
and the vector $b_T$ is a weight vector of weight $\operatorname{wt}(T)$ (see \eqref{eq: wt(T)} for the definition):
\[
  b_T \in V(\lambda)_{\operatorname{wt}(T)}.
\]
The canonical basis forms an almost orthonormal basis.

\begin{ex}\label{ex: Umod str of Vvp1 and Vvp2n-1}\normalfont\hfill
  \begin{enumerate}
    \item The $\mathbf{U}$-module structure of $V(\varpi_1)$ is as follows:
    \begin{align*}
      \begin{split}
        &E_i b_{(a)} = \delta_{a,i+1} b_{(i)}, \\
        &F_i b_{(a)} = \delta_{a,i} b_{(i+1)}, \\
        &D_k b_{(a)} = q^{\delta_{a, k}} b_{(a)}.
      \end{split}
    \end{align*}
    \item The $\mathbf{U}$-module structure of $V(\varpi_{2n-1})$ is as follows (recall the map $\cdot^\vee$ from \eqref{eq: a^vee}):
    \begin{align*}
      \begin{split}
        &E_i b_{a^\vee} = \delta_{a,i} b_{(i+1)^\vee}, \\
        &F_i b_{a^\vee} = \delta_{a,i+1} b_{i^\vee}, \\
        &D_k b_{a^\vee} = q^{1-\delta_{a, k}} b_{a^\vee}.
      \end{split}
    \end{align*}
  \end{enumerate}
\end{ex}

\begin{prop}[{{\it cf.}\ \cite{Kw09}}]\label{prop: quantum maps}\hfill
  \begin{enumerate}
    \item Let $l \in [1,2n]$.
    Then, there exists an injective $\mathbf{U}$-module homomorphism
    \[
      \bigvee: V(\varpi_l) \rightarrow V(\varpi_1) \otimes V(\varpi_{l-1})
    \]
    such that
    \[
      \bigvee(b_{(a_1,\dots,a_l)}) \equiv b_{(a_l)} \otimes b_{(a_1,\dots,a_{l-1})}
    \]
    for all $(a_1,\dots,a_l) \in \SST_{2n}(\varpi_l)$.

    \item Let $k' \in [0,2n-1]$.
    Then, there exists a $\mathbf{U}$-module homomorphism
    \[
      \bigwedge: V(\varpi_1) \otimes V(\varpi_{k'}) \rightarrow V(\varpi_{k'+1})
    \]
    such that
    \[
      \bigwedge(b_{(a)} \otimes b_{(a_1,\dots,a_{k'})}) \equiv \begin{cases}
        b_{(a_1,\dots,a_{k'},a)} & \text{ if } a > a_{k'}, \\
        0 & \text{ if } a \leq a_{k'}.
      \end{cases}
    \]
    for all $a \in [1,2n]$ and $(a_1,\dots,a_{k'}) \in \SST_{2n}(\varpi_{k'})$.
    
    \item Let $k,l \in [0,2n]$.
    Then, there exists a $\mathbf{U}$-module isomorphism
    \[
      R = R_{k,l}: V(\varpi_k) \otimes V(\varpi_l) \rightarrow V(\varpi_l) \otimes V(\varpi_k)
    \]
    such that
    \[
      R(b_{\mathbf{a}} \otimes b_{\mathbf{b}}) \equiv b_{\mathbf{c}} \otimes b_{\mathbf{d}}
    \]
    for all $\mathbf{a} \in \SST_{2n}(\varpi_k)$ and $\mathbf{b} \in \SST_{2n}(\varpi_l)$, where
    \[
      (\mathbf{c}, \mathbf{d}) := R(\mathbf{a}, \mathbf{b}).
    \]

    \item Let $l \in [0,2n]$ and $\mu \in \Par_{\leq 2n}$.
    Then, there exists a $\mathbf{U}$-module isomorphism
    \[
      *: V(\varpi_l) \otimes V(\mu) \rightarrow \bigoplus_{\substack{\lambda \in \Par_{\leq 2n} \\ \mu \underset{\text{vert}}{\subseteq} \lambda}} V(\lambda)
    \]
    such that
    \[
      *(b_S \otimes b_T) \equiv b_{S*T}
    \]
    for all $(S, T) \in \SST_{2n}(\varpi_l) \times \SST_{2n}(\mu)$.

    \item Let $\lambda \in \Par_{\leq 2n}$, and set $l := \ell(\lambda)$.
    Let $\lambda'$ denote the partition $(\lambda_1-1,\dots,\lambda_l-1)$.
    Then, there exists a $\mathbf{U}$-module homomorphism
    \[
      d: V(\lambda) \rightarrow V(\varpi_l) \otimes V(\lambda')
    \]
    such that
    \[
      d(b_T) \equiv b_{\mathbf{a}} \otimes b_{T'}
    \]
    for all $T \in \SST_{2n}(\lambda)$, where
    \[
      (\mathbf{a}, T') := d(T).
    \]
  \end{enumerate}
\end{prop}

\subsection{Quantum symmetric pair of type $A\mathrm{II}$}
For each $i \in [1,n]$, set
\[
  B_{2i} := F_{2i} -q [E_{2i-1}, [E_{2i+1}, E_{2i}]_{q^{-1}}]_{q^{-1}} K_{2i}^{-1} \in \mathbf{U}.
\]
Here, $[,]_{q^{-1}}$ denotes the $q$-commutator given by
\[
  [x,y]_{q^{-1}} := xy - q^{-1} yx.
\]
Let $\mathbf{U}^\imath$ denote the subalgebra of $\mathbf{U}$ generated by
\[
  \{ E_{2i-1}, F_{2i-1}, K_{2i-1}^{\pm 1} \mid i \in [1,n] \} \sqcup \{ B_{2i} \mid i \in [1,n-1] \}.
\]
The pair $(\mathbf{U}, \mathbf{U}^\imath)$ forms a quantum symmetric pair of type $A\mathrm{II}_{2n-1}$.

A $\mathbf{U}^\imath$-module $M$ is said to be a \emph{weight module} if it admits a decomposition
\[
  M = \bigoplus_{\mathbf{z} \in \mathbb{Z}^{n}} M_\mathbf{z}
\]
as a vector space such that
\[
  M_\mathbf{z} = \{ m \in M \mid K_{2i-1} m = q^{z_{2i-1}-z_{2i}} m \ \text{ for all } i \in [1,n] \}.
\]

The character of a finite-dimensional weight $\mathbf{U}^\imath$-module $M$ is the Laurent polynomial $\operatorname{ch}_\imath M \in \mathbb{Z}[y_1^{\pm 1}, \dots, y_{n}^{\pm 1}]$ given by
\[
  \operatorname{ch}_{\imath} M = \sum_{\mathbf{z} \in \mathbb{Z}^{n}} (\dim M_\mathbf{z}) \mathbf{y}^\mathbf{z}.
\]

\begin{prop}\label{prop: ch_i M}
  Let $M$ be a finite-dimensional weight $\mathbf{U}$-module.
  Then, we have
  \[
    \operatorname{ch}_\imath M = \res^\mathfrak{k}(\operatorname{ch} M).
  \]
\end{prop}

For each $\nu \in \Par_{\leq n}$, there exists, up to isomorphism, a unique finite-dimensional irreducible $\mathbf{U}^\imath$-module $V^\imath(\nu)$ such that
\[
  \operatorname{ch}_{\imath} V^\imath(\nu) = s^{Sp}_\nu(y_1,\dots,y_n)
\]
({\it cf.}\ \cite[Proposition 3.3.9 and Corollary 4.3.2]{W21a}).

The anti-automorphism $\rho$ on $\mathbf{U}$ restricts to an anti-automorphism on $\mathbf{U}^\imath$.
The notions of contragredient inner product, almost orthogonal basis, and binary relation $\equiv$ are defined for $\mathbf{U}^\imath$-modules as in the same way as $\mathbf{U}$-modules.

Let $M$ be a $\mathbf{U}$-module equipped with a contragredient inner product.
When we regard the $\mathbf{U}$-module $M$ as a $\mathbf{U}^\imath$-module by restriction, the inner product is still contragredient.
In particular, the irreducible $\mathbf{U}$-module $V(\lambda)$, regarded as a $\mathbf{U}^\imath$-module, admits a contragredient inner product.
Hence, it is completely reducible.

\begin{prop}\label{prop: q-multiplicity}
  Let $\lambda \in \Par_{\leq 2n}$.
  Then, the multiplicity of $V^\imath(\nu)$ in $V(\lambda)$ coincides with $m_{\lambda,\nu}$.
\end{prop}

\begin{proof}
  Let us write
  \[
    V(\lambda) \simeq \bigoplus_{\nu \in \Par_{\leq n}} V^\imath(\nu)^{m'_{\lambda,\nu}}
  \]
  for some $m'_{\lambda,\nu} \geq 0$.
  Then, we have
  \[
    \sum_{\nu \in \Par_{\leq n}} m'_{\lambda,\nu} s^{Sp}_\nu(y_1,\dots,y_n) = \operatorname{ch}_\imath V(\lambda) = \res^\mathfrak{k}(s_\lambda(x_1,\dots,x_{2n})) = \operatorname{ch}_\mathfrak{k} V^\mathfrak{g}(\lambda),
  \]
  where the second equality follows from Proposition \ref{prop: ch_i M}.
  Now, the assertion follows from Proposition \ref{prop: ch_t V^g(lm)} and the linearly independence of the symplectic Schur functions.
\end{proof}

\begin{lem}\label{lem: quantum K}
  There exists a $\mathbf{U}^\imath$-module isomorphism
  \[
    K: V(\varpi_1) \rightarrow V(\varpi_{2n-1})
  \]
  such that
  \[
    K(b_{(a)}) = b_{s(a)^\vee} \ \text{ for all } a \in [1,2n].
  \]
\end{lem}

\begin{proof}
  Clearly, the linear map $K$ defined as above is an isomorphism.
  Using Example \ref{ex: Umod str of Vvp1 and Vvp2n-1}, one can straightforwardly verify that
  \[
    K(x b_{(a)}) = x b_{s(a)^\vee} \ \text{ for all } x \in \mathbf{U}^\imath \text{ and } a \in [1,2n].
  \]
  Hence, the map $K$ is a $\mathbf{U}^\imath$-module isomorphism.
\end{proof}

\begin{lem}\label{lem: quantum pi}
  There exists a $\mathbf{U}^\imath$-module homomorphism
  \[
    \pi: V(\varpi_2) \rightarrow V(\varpi_2) \oplus V(\varpi_0)
  \]
  such that
  \[
    \pi(b_{(a_1,a_2)}) \equiv \begin{cases}
      b_{()} & \text{ if } (a_1,a_2) = (1,2), \\
      b_{(a_1,a_2)} & \text{ if } (a_1,a_2) \neq (1,2).
    \end{cases}
  \]
  for all $1 \leq a_1 < a_2 \leq 2n$.
\end{lem}

\begin{proof}
  The assertion follows from \cite[Theorem 4.3.1]{W23+}.
  Alternatively, one can verify that the vector $w_0 := b_{(1,2)} - q^{-2} b_{(3,4)} \in V(\varpi_2)$ spans the $\mathbf{U}^\imath$-submodule isomorphic to $V(\varpi_0)$ ({\it cf.}\ \cite[Lemma 4.1.4]{W23}).
\end{proof}

\subsection{Reduction map and successor map}
\begin{defi}\normalfont
  The \emph{reduction map} is the $\mathbf{U}^\imath$-module homomorphism
  \[
    \red = \red_l : V(\varpi_l) \rightarrow \bigoplus_{\substack{0 \leq k \leq \min(l, 2n-l) \\ l-k \in 2\mathbb{Z}}} V(\varpi_k)
  \]
  defined inductively as follows.
  \begin{enumerate}
    \item If $l \leq 1$, then the reduction map is the identity map.
    \item If $l > 1$, then the reduction map is the composite
    \[
      \pi \circ \bigwedge \circ (K^{-1} \otimes \mathrm{id}) \circ R \circ (\red_{l-1} \otimes \mathrm{id}) \circ R \circ (K \otimes \mathrm{id}) \circ \bigvee
    \]
  \end{enumerate}
\end{defi}

\begin{prop}\label{prop: quantum red is comb red at infty}
  For each $\mathbf{a} \in \SST_{2n}(\varpi_l)$, we have
  \[
    \red(b_{\mathbf{a}}) \equiv b_{\red(\mathbf{a})}.
  \]
\end{prop}

\begin{proof}
  The assertion follows from Corollary \ref{cor: factorization of red} \eqref{item: factorization of red 1}, Proposition \ref{prop: quantum maps}, and Lemmas \ref{lem: quantum K} and \ref{lem: quantum pi}.
\end{proof}

\begin{defi}\normalfont
  Let $\lambda = (\lambda_1,\dots,\lambda_l) \in \Par_{\leq 2n}$.
  Set $\lambda' := (\lambda_1-1,\dots,\lambda_l-1)$.
  The \emph{successor map} is the $\mathbf{U}^\imath$-module homomorphism
  \[
    \suc : V(\lambda) \rightarrow \bigoplus_{\lambda' \underset{\text{vert}}{\subseteq} \mu} V(\mu)
  \]
  defined to be the composite
  \[
    \suc := * \circ (\red \otimes \mathrm{id}) \circ d.
  \]
\end{defi}

\begin{prop}
  Let $\lambda \in \Par_{\leq 2n}$ and $T \in \SST_{2n}(\lambda)$.
  Then, we have
  \[
    \suc(b_T) \equiv b_{\suc(T)}.
  \]
\end{prop}

\begin{proof}
  The assertion follows from Propositions \ref{prop: quantum maps} and \ref{prop: quantum red is comb red at infty}.
\end{proof}

\subsection{An orthonormal basis of $V^\imath(\nu)$}

\begin{prop}\label{prop: p_nu}
  Let $\nu \in \Par_{\leq n}$.
  Then, there exists a basis $\{ b^\imath_T \mid T \in \SpT_{2n}(\nu) \}$ of $V^\imath(\nu)$ and a $\mathbf{U}^\imath$-module homomorphism
  \[
    p_\nu: V(\nu) \rightarrow V^\imath(\nu)
  \]
  such that
  \[
    p_\nu(b_T) \equiv \begin{cases}
      b^\imath_T & \text{ if } T \in \SpT_{2n}(\nu), \\
      0 & \text{ if } T \notin \SpT_{2n}(\nu),
    \end{cases}
  \]
  for all $T \in \SST_{2n}(\nu)$.
\end{prop}

\begin{proof}
  Let $p_0 : V(\nu) \rightarrow V(\nu)$ and $p_1: V(\nu) \rightarrow \bigoplus_{\xi \underset{\text{vert}}{\subset} \nu} V(\xi)$ denote the composite of the successor map $\suc: V(\nu) \rightarrow \bigoplus_\xi V(\xi)$ followed by the projections, respectively.
  Then, by Corollary \ref{cor: suc(T) = T}, we have
  \[
    p_0(b_T) \equiv \begin{cases}
      b_T & \text{ if } T \in \SpT_{2n}(\nu), \\
      0 & \text{ if } T \notin \SpT_{2n}(\nu),
    \end{cases} \quad p_1(b_T) \equiv \begin{cases}
      0 & \text{ if } T \in \SpT_{2n}(\nu), \\
      b_{\suc(T)} & \text{ if } T \notin \SpT_{2n}(\nu),
    \end{cases}
  \]
  This implies that $p_0(V(\nu))$ and $p_1(V(\nu))$ contain linearly independent subsets $\{ p_0(b_T) \mid T \in \SpT_{2n}(\nu) \}$ and $\{ p_1(b_T) \mid T \notin \SpT_{2n}(\nu) \}$, respectively.
  Since the successor map on $\SST_{2n}(\nu)$ is injective by Corollary \ref{cor: inj of comb suc}, we see that they are bases of the two spaces.

  Now, we compute as
  \[
    \operatorname{ch}_\imath p_0(V(\nu)) = \sum_{T \in \SpT_{2n}(\nu)} \mathbf{y}^{\operatorname{wt}^{Sp}(T)} = \operatorname{ch}_\imath V^\imath(\nu).
  \]
  This implies $p_0(\nu) \simeq V^\imath(\nu)$.
  Thus, we complete the proof.
\end{proof}

\subsection{Quantum Littlewood-Richardson map}
\begin{defi}\normalfont
  The quantum Littlewood-Richardson map is the $\mathbf{U}^\imath$-module homomorphism
  \[
    \LRAII: V(\lambda) \rightarrow \bigoplus_{\substack{\nu \in \Par_{\leq n} \\ \nu \subseteq \lambda}} V^\imath(\nu) \otimes \mathbb{Q}(q) \Rec_{2n}(\lambda/\nu)
  \]
  defined to be the composite of the $\mathbf{U}^\imath$-module homomorphism
  \[
    V(\lambda) \rightarrow \bigoplus_{\substack{\nu \in \Par_{\leq n} \\ \nu \subseteq \lambda}} V(\nu) \otimes \mathbb{Q}(q) \SST_{2n}(\lambda)
  \]
  which sends $b_T$ to $\suc^k(b_T) \otimes T$ for each $T \in \SST_{2n}(\lambda)$, where $k > 0$ is such that $\suc^k(S) = \PAII(S)$ for all $S \in \SST_{2n}(\lambda)$, and the sum of $\mathbf{U}^\imath$-homomorphisms of the form
  \[
    V(\nu) \otimes \mathbb{Q}(q) \SST_{2n}(\lambda) \rightarrow V^\imath(\nu) \otimes \mathbb{Q}(q) \Rec_{2n}(\lambda/\nu)
  \]
  which sends $b_S \otimes T$ to $\delta_{S, \PAII(T)} p_\nu(b_{\PAII(T)}) \otimes \QAII(T)$ for each $S \in \SST_{2n}(\nu)$ and $T \in \SST_{2n}(\lambda)$.
\end{defi}

The following is immediate from the definition.

\begin{prop}\label{prop: quantum LR is comb LR at q infty}
  Let $T \in \SST_{2n}(\lambda)$.
  Then, we have
  \[
    \LRAII(b_T) \equiv b^\imath_{\PAII(T)} \otimes \QAII(T).
  \]
\end{prop}

\section{Surjectivity and codomain of the Littlewood-Richardson map}\label{sect: LR}
In this section, we complete the proof of the first assertion of Theorem \ref{thm: main}.

\subsection{Injectivity and codomain of the Littlewood-Richardson map}
\begin{prop}\label{prop: codomain of LR}
  Let $\lambda \in \Par_{\leq 2n}$ and $T \in \SST_{2n}(\lambda)$.
  Then, the tableau $\PAII(T)$ is symplectic.
  Consequently, we have
  \[
    \LRAII(T) \in \SpT_{2n}(\nu) \times \Rec_{2n}(\lambda/\nu)
  \]
  for some $\nu \in \Par_{\leq n}$ such that $\nu \subseteq \lambda$.
\end{prop}

\begin{proof}
  Since $\suc(\PAII(T)) = \PAII(T)$, it is symplectic by Corollary \ref{cor: suc(T) = T}.
\end{proof}

Let $\lambda,\mu \in \Par_{\leq 2n}$ be such that $\mu \subseteq \lambda$.
Define a map
\[
  \widetilde{\suc}: \SST_{2n}(\mu) \times \Tab(\lambda/\mu) \rightarrow \bigsqcup_{\substack{\mu' \in \Par_{\leq 2n} \\ \mu' \underset{\text{vert}}{\subseteq} \mu}} \SST_{2n}(\mu') \times \Tab(\lambda/\mu')
\]
as follows.
Let $(S, R) \in \SST_{2n}(\mu) \times \Tab(\lambda/\mu)$.
Set
\[
  k := \max \{ R(i,j) \mid (i,j) \in D(\lambda/\mu) \},
\]
and $\mu' := \sh(\suc(S))$.
Then,
\[
  \widetilde{\suc}(S, R) = (\suc(S), R'),
\]
where $R' \in \Tab(\lambda/\mu')$ is such that
\[
  R'(i,j) = \begin{cases}
    R(i,j) & \text{ if } (i,j) \notin D(\mu), \\
    k+1 & \text{ if } (i,j) \in D(\mu),
  \end{cases} \ \text{ for all } (i,j) \in D(\lambda/\mu').
\]

\begin{lem}\label{lem: injectivity of suctil}
  Let $\lambda,\mu \in \Par_{\leq 2n}$ be such that $\mu \subseteq \lambda$.
  Then, the map $\widetilde{suc}$ on $\SST_{2n}(\mu) \times \Tab(\lambda/\mu)$ is injective.
\end{lem}

\begin{proof}
  Let $(S_1, R_1), (S_2, R_2) \in \SST_{2n}(\mu) \times \Tab(\lambda/\mu)$ be such that
  \[
    (\suc(S_1), R'_1) := \widetilde{\suc}(S_1, R_1) = \widetilde{\suc}(S_2, R_2) = (\suc(S_2), R'_2).
  \]
  Then, for all $(i,j) \in D(\lambda/\mu)$, we have
  \[
    R_1(i,j) = R'_1(i,j) = R'_2(i,j) = R_2(i,j).
  \]
  This implies that
  \[
    R_1 = R_2.
  \]
  
  Next, we show that $S_1 = S_2$.
  Let us write
  \[
    d(S_i) = (\mathbf{a}_i, S'_i) \ \text{ for each } i = 1,2.
  \]
  Then, we have
  \[
    \red(\mathbf{a}_1) * S'_1 = \suc(S_1) = \suc(S_2) = \red(\mathbf{a}_2) * S'_2.
  \]
  Since $|\sh(S'_1)| = |\sh(S'_2)|$, we must have
  \[
    |\red(\mathbf{a}_1)| = |\red(\mathbf{a}_2)|.
  \]
  Then, by Proposition \ref{prop: Pieri formula}, we obtain
  \[
    (\red(\mathbf{a}_1), S'_1) = (\red(\mathbf{a}_2), S'_2).
  \]
  Now, Corollary \ref{cor: factorization of red} \eqref{item: factorization of red 2} implies
  \[
    \mathbf{a}_1 = \mathbf{a}_2.
  \]
  Hence, we deduce
  \[
    S_1 = \mathbf{a}_1 * S'_1 = \mathbf{a}_2 * S'_2 = S_2,
  \]
  as desired.
\end{proof}

\begin{prop}\label{prop: inj of LR}
  Let $\lambda \in \Par_{\leq 2n}$.
  Then, the Littlewood-Richardson map on $\SST_{2n}(\lambda)$ is injective.
\end{prop}

\begin{proof}
  Let $T_1, T_2 \in \SST_{2n}(\lambda)$ be such that $\LRAII(T_1) = \LRAII(T_2)$.
  Let us write
  \[
    \LRAII(T_1) = (P, Q),
  \]
  and set $\nu := \sh(P)$.

  First, note that
  \[
    \LRAII(T_1) = \widetilde{\suc}^{k_0}(T_1, Q^0)
  \]
  for a sufficiently large $k_0 > 0$, where $Q^0$ denotes the unique tableau of shape $\lambda/\lambda$.
  For each $k \in [0,k_0]$, let $\nu^k$ denote the partition whose Young diagram is
  \[
    D(\nu^k) = D(\nu) \sqcup \{ (i,j) \in D(\lambda/\nu) \mid Q(i,j) > k \}.
  \]
  Then, we see that
  \[
    \sh(\suc^k(T_1)) = \nu^k \ \text{ for all } k \in [0,k_0].
  \]
  Since $\LRAII(T_2) = \LRAII(T_1) = (P,Q)$, we have also
  \[
    \sh(\suc^k(T_2)) = \nu^k \ \text{ for all } k \in [0,k_0].
  \]
  Applying Lemma \ref{lem: injectivity of suctil}, we deduce that
  \[
    \suc^k(T_1) = \suc^k(T_2) \ \text{ for all } k \in [0,k_0]
  \]
  by descending induction on $k$.
  In particular, we obtain
  \[
    T_1 = \suc^0(T_1) = \suc^0(T_2) = T_2,
  \]
  as desired.
\end{proof}

\subsection{Surjectivity of the Littlewood-Richardson map}
\begin{thm}\label{thm: quantum LR is isom}
  The quantum Littlewood-Richardson map on $V(\lambda)$ is a $\mathbf{U}^\imath$-module isomorphism.
\end{thm}

\begin{proof}
  Let $Q \in \Rec_{2n}(\lambda/\nu)$.
  Then, there exists $T \in \SST_{2n}(\lambda)$ such that $\QAII(T) = Q$.
  By Proposition \ref{prop: quantum LR is comb LR at q infty}, we have
  \[
    \LRAII(b_T) \equiv b_{\PAII(T)} \otimes Q \neq 0.
  \]
  Since the summand $V^\imath(\nu) \otimes \mathbb{Q}(q) Q$ is an irreducible $\mathbf{U}^\imath$-module, we see that it is contained in the image of $\LRAII$.
  Since the recording tableau $Q$ is arbitrary, the $\LRAII$ on $V(\lambda)$ is surjective.

  The injectivity follows from that of the Littlewood-Richardson map on $\SST_{2n}(\lambda)$ and Proposition \ref{prop: quantum LR is comb LR at q infty}
\end{proof}

\begin{cor}\label{cor: surj of LR}
  The Littlewood-Richardson map on $\SST_{2n}(\lambda)$ is bijective.
\end{cor}

\section{Characterization of the recording tableaux}\label{sect: rec}
In this section, we prove the second assertion of Theorem \ref{thm: main}.

\subsection{Some properties of the successor map}
Let $\lambda = (\lambda_1,\dots,\lambda_l) \in \Par_{\leq 2n}$ and $T \in \SST_{2n}(\lambda)$.
Let us write $d(T) = (\mathbf{a}, S)$, $\mathbf{a} = (a_1,\dots,a_l)$, and $\red(\mathbf{a}) = \mathbf{b} = (b_1,\dots,b_k)$.

For each $t \in [0,k]$, set
\[
  S_t := \begin{cases}
    S & \text{ if } t = 0, \\
    b_t \rightarrow S_{t-1} & \text{ if } t > 0,
  \end{cases}
\]
and $\lambda^t := \sh(S_t)$.
Let us write
\[
  \br(b_t, S_{t-1}) = (r_{t,1}, r_{t,2}, \dots, r_{t,s_t}).
\]
Also, let $r_{t,0} \in [1,l]$ be such that
\[
  a_{r_{t,0}} = b_t.
\]
Set
\[
  \mathrm{br}_{\leq t} := \{ (r_{u,j},j) \mid u \in [1,t],\ j \in [1,s_u] \}.
\]

Set $T' := \suc(T)$, $\mu := \sh(T')$, and $l' := \ell(\mu)$.

Define $\mathbf{a}' = (a'_1,\dots,a'_{l'})$, $S'$, $\mathbf{b}' = (b'_1,\dots,b'_{k'})$, $S'_{t'}$, $\mu^{t'}$, $(r'_{t',1}, \dots, r'_{t',s'_{t'}})$ for $t' \in [1,k']$ in the same way as before.

\begin{lem}\label{lem: 1}
  Let $t \in [1,k]$ be such that $r_{t,0} = r_{t,1}$ and $a'_{r_{t,1}+1} = a'_{r_{t,1}} + 1$.
  Then, we have $a_{r_{t,0}+1} = a_{r_{t,0}}+1$.
\end{lem}

\begin{proof}
  By Lemma \ref{lem: properties of suc(T)} \eqref{item: properties of suc(T) 2}, we have
  \[
    a'_{r_{t,1}} = T'(r_{t,1}, 1) = T(r_{t,0}, 1) = a_{r_{t,0}}.
  \]

  First, suppose that $(r_{t,1}+1, 1) \notin \mathrm{br}_{\leq k}$.
  Then, we have
  \[
    T'(r_{t,1}+1, 1) = T(r_{t,1}+1, 2).
  \]
  Hence, we deduce that
  \[
    a'_{r_{t,1}+1} = T'(r_{t,1}+1, 1) = T(r_{t,1}+1, 2) \geq T(r_{t,1}+1, 1) = a_{r_{t,1}+1}.
  \]
  This implies that
  \[
    a_{r_{t,0}+1} = a_{r_{t,1}+1} \leq a'_{r_{t,1}+1} = a'_{r_{t,1}}+1 = a_{r_{t,0}}+1.
  \]
  Therefore, the assertion follows.

  Next, suppose that $(r_{t,1}+1, 1) \in \mathrm{br}_{\leq k}$.
  Let us write $r_{t,1}+1 = r_{u,1}$ for some $u \in [1,k]$.
  Then, we have
  \[
    T'(r_{t,1}+1, 1) = T'(r_{u,1}, 1) = T(r_{u,0}, 1).
  \]
  This implies that
  \[
    a'_{r_{t,1}+1} = T'(r_{t,1}+1, 1) = T(r_{u,0}, 1) = a_{r_{u,0}}.
  \]
  Since we have
  \[
    r_{1,1} < r_{2,1} < \cdots < r_{k,1}
  \]
  by Lemma \ref{lem: properties of suc(T)} \eqref{item: properties of suc(T) 2.2}, 
  it must hold that $u > t$.
  Noting that $a_{r_{u,0}} > a_{r_{t,0}}$, we obtain as before that
  \[
    a_{r_{u,0}} = a_{r_{t,0}}+1.
  \]
  This implies that $r_{u,0} = r_{t,0}+1$ since $(a_1,\dots,a_k)$ is a strictly increasing sequence.
  Hence, the assertion follows.
\end{proof}

\begin{lem}\label{lem: N(t)}
  Let $t \in [0,k]$.
  Then, we have
  \[
    \sharp \{ t' \in [1,k'] \mid r'_{t',0} \leq r_{t,1} \} \geq t,
  \]
  where we set $r_{0,1} := 0$.
  In particular, it holds that $k' \geq k$.
\end{lem}

\begin{proof}
  For each $t \in [0,k]$, set
  \[
    N(t) := \sharp \{ t' \in [1,k'] \mid r'_{t',0} \leq r_{t,1} \}.
  \]
  Note that we have
  \[
    N(t) = \sharp \{ i \in [1,r_{t,1}] \mid a'_i \notin \Rem(\mathbf{a}') \} = r_{t,1} - |(a'_1,\dots,a'_{r_{t,1}}) \cap \Rem(\mathbf{a}')|.
  \]

  Let us prove the assertion by induction on $t$.
  When $t = 0$, the assertion is trivial.
  Assume that $t > 0$ and the assertion holds for $0,1,\dots,t-1$.

  First, suppose that $a'_{r_{t,1}} \notin \Rem(\mathbf{a}')$.
  Then, we have
  \[
    N(t) \geq N(t-1) + 1.
  \]
  Hence, we deduce from our induction hypothesis that
  \[
    N(t) \geq N(t-1) + 1 \geq (t-1) + 1 = t,
  \]
  as desired.

  Next, suppose that $a'_{r_{t,1}} \in \Rem(\mathbf{a}')$ and $a'_{r_{t,1}} \notin 2\mathbb{Z}$.
  Then, by Proposition \ref{prop: characterization of rem(a)}, we must have $r_{t,1} < l$, $a'_{r_{t,1}+1} = a'_{r_{t,1}}+1$, and
  \[
    a'_{r_{t,1}} < 2r_{t,1} - |\Rem(a'_1,\dots,a'_{r_{t,1}-1})|.
  \]
  Also, by Corollary \ref{cor: low bound for rem(a')}, we have
  \[
    |\Rem(a'_1,\dots,a'_{r_{t,1}-1})| \geq 2r_{t,1} - a'_{r_{t,1}} - 1.
  \]
  Hence, we obtain
  \[
    |\Rem(a'_1,\dots,a'_{r_{t,1}-1})| = 2r_{t,1} - a'_{r_{t,1}} - 1.
  \]
  Therefore,
  \[
    N(t) = r_{t,1} - (|\Rem(a'_1,\dots,a'_{r_{t,1}-1})|+1) = a'_{r_{t,1}} - r_{t,1}.
  \]

  On the other hand, by Corollary \ref{cor: low bound for rem(a')} again, we have
  \[
    |\Rem(a_1,\dots,a_{r_{t,0}-1})| \geq 2r_{t,0} - a_{r_{t,0}} - 1.
  \]
  Also, since $a_{r_{t,0}} \notin \Rem(\mathbf{a})$, it holds that
  \[
    r_{t,0} - |\Rem(a_1,\dots,a_{r_{t,0}-1})| = t.
  \]
  Hence,
  \[
    a_{r_{t,0}} - r_{t,0} \geq t-1.
  \]

  Combining above, we obtain
  \[
    N(t) = a'_{r_{t,1}} - r_{t,0} + (r_{t,0} - r_{t,1}) \geq t + (r_{t,0} - r_{t,1}) - 1.
  \]
  Therefore, the assertion follows when $r_{t,0} > r_{t,1}$.
  Hence, we only need to consider the case when $r_{t,0} = r_{t,1}$.
  In this case, Lemma \ref{lem: 1} implies
  \[
    a_{r_{t,0}+1} = a_{r_{t,0}}+1.
  \]
  Since $a_{r_{t,0}} \notin \Rem(\mathbf{a})$, we have
  \[
    a_{r_{t,0}} \geq 2r_{t,0} - |\Rem(a_1,\dots,a_{r_{t,0}-1})| = r_{t,0} + t.
  \]
  Now, we deduce
  \[
    N(t) = a'_{r_{t,1}} - r_{t,1} = a_{r_{t,0}} - r_{t,0} \geq t,
  \]
  as desired.

  Finally, suppose that $a'_{r_{t,1}} \in \Rem(\mathbf{a}')$ and $a'_{r_{t,1}} \in 2\mathbb{Z}$.
  In this case, one can deduce the assertion in a similar way to the previous case.

  Thus, we complete the proof.
\end{proof}

\begin{lem}
  Let $t \in [0,k]$.
  Then, we have $s'_t \geq s_t$ and $r'_{t,j-1} \leq r_{t,j}$ for all $j \in [1,s_t]$, where we set $s_0 = s'_0 = 0$.
\end{lem}

\begin{proof}
  We proceed by induction on $t$.
  The case when $t = 0$ is trivial.
  Hence, assume that $t > 0$ and the assertion holds for $0,1,\dots,t-1$.

  We prove the assertion by induction on $j$.
  By lemma \ref{lem: N(t)} and the fact that $r'_{1,0} < r'_{2,0} < \cdots < r'_{k',0}$, we obtain
  \[
    r'_{t,0} \leq r_{t,1}.
  \]

  Now, assume that $j > 1$ and we have $s'_t, s_t \geq j-1$ and $r'_{t,j-2} \leq r_{t,j-1}$.
  When $s_t = j-1$, there is nothing to prove.
  Hence, assume that $s_t \geq j$.

  By Lemma \ref{lem: properties of suc(T)}, we have
  \[
    r'_{t,j-1} = \min \{ r' \in [1,\col_{j-1}(\mu^{t-1})+1] \mid T'(r',j) \geq T'(r'_{t,j-2}, j-1) \}.
  \]
  Note that our induction hypothesis implies that
  \[
    r'_{t-1,j-1} \leq r_{t-1,j}.
  \]
  The right-hand side is less than $r_{t,j}$ by Lemma \ref{lem: properties of suc(T)} \eqref{item: properties of suc(T) 2.2}.
  Also we have
  \[
    r_{t,j} \leq \col_j(\lambda^{t-1}) + 1,
  \]
  and
  \[
    \col_j(\lambda^{t-1}) \leq \col_j(\lambda^k) = \col_j(\mu^0) \leq \col_{j-1}(\mu^0) \leq \col_{j-1}(\mu^{t-1}).
  \]
  Observe that
  \begin{align*}
    \begin{split}
      &T'(r_{t,j},j) = T(r_{t,j-1}, j), \\
      &T'(r'_{t,j-2},j-1) = \begin{cases}
        T(r_{u,j-2},j-1) & \text{ if } r'_{t,j-2} = r_{u,j-1} \text{ for some } u \in [1,k], \\
        T(r'_{t,j-2}, j) & \text{ otherwise},
      \end{cases}
    \end{split}
  \end{align*}
  and
  \begin{align*}
    \begin{split}
      &T(r_{u,j-2},j-1) \leq T(r_{u,j-1},j) \leq T(r_{t,j-1},j), \\
      &T(r'_{t,j-2},j) \leq T(r_{t,j-1},j).
    \end{split}
  \end{align*}
  By above, we obtain $r'_{t,j-1} \leq r_{t,j}$ and $s'_t \geq j$, as desired.
\end{proof}

\begin{prop}\label{prop: for cond e}
  Let $r \in [0,l]$.
  Then, we have
  \[
    \sharp \{ t \in [1,k] \mid r_{t,s_t} \leq r \} \leq \sharp \{ t' \in [1,k'] \mid r'_{t', s_{t'}} \leq r \}.
  \]
\end{prop}

\begin{proof}
  For each $t \in [1,k]$, we have $s'_t \geq s_t$ and $r'_{t,s_t-1} \leq r_{t,s_t}$.
  Hence, if $r_{t,s_t} \leq r$, then
  \[
    r'_{t,s'_t} \leq r'_{t,s_t-1} \leq r_{t,s_t} \leq r.
  \]
  Thus, the assertion follows.
\end{proof}

\subsection{Symplectic Littlewood-Richardson tableaux}
\begin{defi}\label{def: spLR}\normalfont
  Let $\lambda \in \Par_{\leq 2n}$ and $\nu \in \Par_{\leq n}$.
  A tableau $T \in \Tab(\lambda/\nu)$ is said to be a \emph{symplectic Littlewood-Richardson tableau} if it satisfies the following.
  \begin{enumerate}
    \item\label{item: def spLR 1} $T \in \SST_{2n}(\lambda/\nu)$.
    \item\label{item: def spLR 2} Let $(w_1,\dots,w_N)$ denote the column-word $w^\mathrm{col}(T)$ of $T$.
    Then, the reversed word $(w_N,\dots,w_1)$ is a lattice permutation:
    For each $r \in [1,N]$ and $k \in [1,2n-1]$, the number of occurrences of $k$ in the subsequence $(w_N, \dots, w_r)$ is greater than or equal to that of $k+1$.
    \item\label{item: def spLR 3} The sequence $\operatorname{wt}(T) = (T[1], T[2], \dots, T[2n])$ is a partition which has even columns (see Definition \ref{def: even col}).
    \item\label{item: def spLR 4} If $T(i,j) = 2k+1$ for some $(i,j) \in D(\lambda/\nu)$ and $k \in \mathbb{Z}_{\geq 0}$, then we have $i \leq n+k$.
  \end{enumerate}
  Let $\LRSp_{2n}(\lambda/\nu)$ denote the set of all symplectic Littlewood-Richardson tableaux of shape $\lambda/\nu$.
\end{defi}

\begin{rem}\normalfont
  Let $T \in \SST_{2n}(\lambda/\nu)$ be such that the sequence $\operatorname{wt}(T)$ is a partition which has even columns.
  Then, we have $T \in \LRSp_{2n}(\lambda/\nu)$ if and only if the reversed column-word \emph{fits $\lambda/\nu$ $n$-symplectically} in the sense of \cite[Definition 3.9]{S90}.
\end{rem}

The symplectic Littlewood-Richardson tableaux provide us a branching rule:

\begin{thm}[{{\it cf.}\ \cite[Corollary 3.12]{S90}}]\label{thm: Sun90}
  Let $\lambda \in \Par_{\leq 2n}$ and $\nu \in \Par_{\leq n}$ be such that $\nu \subseteq \lambda$.
  Then, we have
  \[
    m_{\lambda,\nu} = |\LRSp_{2n}(\lambda/\nu)|.
  \]
\end{thm}

\subsection{Characterization of the recording tableaux}
Let $\lambda \in \Par_{\leq 2n}$, $\nu \in \Par_{\leq n}$ be such that $\nu \subseteq \lambda$.

\begin{lem}\label{lem: rec in rectil}
  We have
  \[
    \Rec_{2n}(\lambda/\nu) \subseteq \widetilde{\Rec}_{2n}(\lambda/\nu).
  \]
\end{lem}

\begin{proof}
  Let $Q \in \Rec_{2n}(\lambda/\nu)$.
  We only need to show that $Q \in \widetilde{\Rec}_{2n}(\lambda/\nu)$, that is, to verify that the tableau $Q$ satisfies conditions (R1)--(R5) in Section \ref{sect: main}.

  By the definition of recording tableaux, there exists $T \in \SST_{2n}(\lambda)$ such that $\QAII(T) = Q$.
  For each $k \in \mathbb{Z}_{\geq 0}$, set
  \[
    P^k := \begin{cases}
      T & \text{ if } k = 0, \\
      \suc(P^{k-1}) & \text{ if } k > 0,
    \end{cases} \quad \nu^k := \sh(P^k).
  \]
  Let us write
  \[
    d(P^k) = (\mathbf{a}^k, S^k),\ \mathbf{a}^k = (a^k_1,\dots,a^k_{l_k}),\ \red(\mathbf{a}^k) = \mathbf{b}^k = (b^k_1,\dots,b^k_{l'_k}),
  \]
  and
  \[
    \mathrm{br}(b^k_t, b^k_{t-1} \rightarrow ( \cdots \rightarrow (b^k_1 \rightarrow S^{k-1}))) = (r^k_{t,1}, r^k_{t,2}, \dots, r^k_{t,s^k_t}).
  \]

  First, let us verify conditions (R1) and (R2).
  Let $(i,j) \in D(\lambda/\nu)$ and write $Q(i,j) = k$.
  Suppose that $(i,j+1) \in D(\lambda/\nu)$ (resp., $(i+1,j) \in D(\lambda/\nu)$).
  Then, since $\nu^k \underset{\text{vert}}{\subseteq} \nu^{k-1}$, we must have $(i,j+1) \notin D(\nu^{k-1})$ (resp., $(i+1,j) \notin D(\nu^k)$).
  This, together with Lemma \ref{lem: QAII, nu}, implies that $Q(i,j+1) \leq k-1$ (resp., $Q(i+1,j) \leq k$), as desired.

  Next, let us verify conditions (R3) and (R4).
  Let $k > 0$.
  Then, we have
  \[
    Q[k] = |\nu^{k-1}/\nu^k|
  \]
  by Lemma \ref{lem: QAII, nu}.
  Since $P^k = \suc(P^{k-1}) = \red(\mathbf{a}^{k-1})*S^{k-1}$, it holds that
  \[
    |\nu^{k-1}/\nu^k| = |\Rem(\mathbf{a}^{k-1})|.
  \]
  The right-hand side is even by Lemma \ref{lem: rem is invariant under switch}, and is greater than or equal to $2(\ell(\nu^{k-1})-n)$ by Proposition \ref{prop: low bound for rem}.

  Finally, let us verify condition (R5).
  Let $r,k > 0$.
  Then, we have
  \begin{align*}
    \begin{split}
      Q_{\leq r}[k] &= \sharp \{ i \in [1,r] \mid Q(i,j) = k \ \text{ for some } j \} \\
      &= |[1,r] \setminus \{ r^{k-1}_{t,s^{k-1}_t} \mid t \in [1,l'_{k-1}] \}| \\
      &= r - \sharp \{ t \in [1,l'_{k-1}] \mid r^{k-1}_{t,s^{k-1}_t} \leq r \}.
    \end{split}
  \end{align*}
  Hence, we deduce that
  \[
    Q_{\leq r}[k] - Q_{\leq r}[k+1] = \{ t' \in [1,l'_{k}] \mid r^{k}_{t',s^{k}_{t'}} \leq r \} - \{ t \in [1,l'_{k-1}] \mid r^{k-1}_{t,s^{k-1}_t} \leq r \}.
  \]
  The right-hand side is nonnegative by Proposition \ref{prop: for cond e}.

  Thus, we complete the proof.
\end{proof}

\begin{lem}\label{lem: rectil in lrt}
  Consider the map
  \[
    \widetilde{\Rec}_{2n}(\lambda/\nu) \rightarrow \Tab(\lambda/\nu)
  \]
  which sends each $R \in \widetilde{\Rec}_{2n}(\lambda/\nu)$ to the tableaux of shape $\lambda/\nu$ whose $(i,j)$ entry is equal to $R_{\leq i}[R(i,j)]$; the number of occurrences of $R(i,j)$ in $R$ in the $i$-th row or above.
  Then, it is injective and its image is contained in $\LRSp_{2n}(\lambda/\nu)$.
\end{lem}

\begin{proof}
  Let $R \in \widetilde{\Rec}_{2n}(\lambda/\nu)$, and $S \in \Tab(\lambda/\nu)$ be the image of $R$ under the map above.
  We only need to show that the tableau $S$ satisfies the conditions \eqref{item: def spLR 1}--\eqref{item: def spLR 4} in Definition \ref{def: spLR}.

  First, let us verify condition \eqref{item: def spLR 1}.
  By condition (R1), each $k' > 0$ appears as an entry of $R$ at most once in each row.
  In particular, we have
  \[
    R[k'] \leq \ell(\lambda) \leq 2n \ \text{ for all } k' > 0.
  \]
  This implies that the entries of $S$ are in $[1,2n]$.

  Let $(i,j) \in D(\lambda/\nu)$ and set $k := R(i,j)$.
  Suppose that $(i,j+1) \in D(\lambda/\nu)$ (resp., $(i+1,j) \in D(\lambda/\nu)$), and set $k' := R(i,j+1) < k$ (resp., $k'' := R(i+1,j) \leq k$).
  Here, we used condition (R1) (resp., (R2)).
  Then, we have
  \[
    S(i,j+1) = R_{\leq i}[k'] \geq R_{\leq i}[k] = S(i,j),
  \]
  (resp., 
  \[
    S(i+1,j) = R_{\leq i+1}[k''] = R_{\leq i}[r''] + 1 \geq R_{\leq i}[r] + 1 = S(i,j) + 1,)
  \]
  where the inequality follows from condition (R5).
  This implies that $S$ is semistandard.

  Second, let us verify condition \eqref{item: def spLR 2}.
  Let us write
  \[
    w_\col(S) = (w_1,\dots,w_N) \text{ and } w_\col(R) = (w'_1,\dots,w'_N).
  \]
  Let $r \in [1,N]$ and $k \in [1,2n]$.
  Then, we have
  \[
    \sharp\{ t \in [r,N] \mid w_t = k \} = \sharp\{ k' > 0 \mid \sharp\{ t' \in [r,N] \mid w'_{t'} = k' \} \geq k \}.
  \]
  The right-hand side decreases as $k$ increases.
  Hence, the assertion follows.

  Next, let us verify condition \eqref{item: def spLR 3}.
  By condition (R5), we have
  \[
    R[1] \geq R[2] \geq \cdots.
  \]
  Then, we can consider the tableau whose $j$-th column consists of exactly $R[j]$ $j$'s.
  Clearly, its shape is $\operatorname{wt}(S)$.
  Noting that $R[j] \in 2\mathbb{Z}$ by condition (R3), we see that the partition $\operatorname{wt}(S)$ has even columns.

  Finally, let us verify condition \eqref{item: def spLR 4}.
  Let $(i,j) \in D(\lambda/\nu)$ and suppose that $S(i,j) = 2k+1$ for some $k \geq 0$.
  When $i \leq n$, there is nothing to prove.
  Hence, assume that $i > n$.

  Set $r := R(i,j)$.
  Then, we have $R_{\leq i}[r] = 2k+1$.
  Let $\mu$ be the partition whose Young diagram is given by
  \[
    D(\mu) = D(\nu) \sqcup \{ (i,j) \in D(\lambda/\nu) \mid R(i,j) \geq r \}.
  \]
  
  We have $R_{\leq i}[r] \geq 2(i-n)$.
  Otherwise, we obtain
  \[
    R[r] = R_{\leq \ell(\mu)}[r] \leq R_{\leq i}[r] + (\ell(\mu)-i) < 2(i-n) + (\ell(\mu)-i) \leq 2(\ell(\mu)-i),
  \]
  which contradicts condition (R4).

  Now, we have
  \[
    2k+1 = R_{\leq i}[r] \geq 2(i-n).
  \]
  This implies that $2k \geq 2(i-n)$, and then the assertion follows.

  Thus, we complete the proof.
\end{proof}

\begin{thm}\label{thm: rec = rectil}
  We have $\Rec_{2n}(\lambda/\nu) = \widetilde{\Rec}_{2n}(\lambda/\nu)$.
\end{thm}

\begin{proof}
  By Lemmas \ref{lem: rec in rectil} and \ref{lem: rectil in lrt}, we have
  \[
    \Rec_{2n}(\lambda/\nu) \subseteq \widetilde{\Rec}_{2n}(\lambda/\nu)
  \]
  and
  \[
    |\Rec_{2n}(\lambda/\nu)| \leq |\widetilde{\Rec}_{2n}(\lambda/\nu)| \leq |\LRSp_{2n}(\lambda/\nu)|.
  \]
  Hence, we only need to show that
  \begin{align}\label{eq: coin card}
    |\Rec_{2n}(\lambda/\nu)| = |\LRSp_{2n}(\lambda/\nu)|.    
  \end{align}

  Recall from Theorem \ref{thm: Sun90} that
  \[
    |\LRSp_{2n}(\lambda/\nu)| = m_{\lambda,\nu}.
  \]
  On the other hand, by Theorem \ref{thm: quantum LR is isom} and Proposition \ref{prop: q-multiplicity}, we obtain
  \[
    |\Rec_{2n}(\lambda/\nu)| = m_{\lambda,\nu}.
  \]
  Therefore, equation \eqref{eq: coin card} holds.

  Thus, we complete the proof.
\end{proof}

\end{document}